\documentclass[dvipdfm,a4paper,11pt,oneside]{article}
\usepackage{amsfonts}
\usepackage{amssymb}
\usepackage{indentfirst}
\usepackage{mathrsfs}
\usepackage{amsfonts,amsmath,latexsym,bm}
\usepackage{amsthm}
\usepackage{float}
\newcommand{\tabcaption}{\def\@captype{table}\caption}
\usepackage{graphicx}
\newtheorem{lem}{Lemma}[section]
\newtheorem{thm}{Theorem}[section]

\newtheorem{rem}{Remark}[section]

\newtheorem{assum}{Assumption}[section]
\newtheorem{define}{Definition}[section]

\numberwithin{equation}{section}

\begin{document}

%%%%%%%%%%%%%%%%%%%%%%%%%%%%%%%%%%%%%%%%%%%%%%%%%%%%%
\title{\Large\bf  Hybrid stress quadrilateral   finite element approximation for  stochastic plane  elasticity equations
\thanks{This work was supported in part by
National Natural Science Foundation of China (11171239),  Major Research
	Plan of  National Natural Science Foundation of China (91430105) and
 Open Fund of  Key Laboratory of Mountain Hazards and Earth Surface Processes, CAS.}}

\author {
 Xiaojing Xu\thanks{Email: xuxiaojing0603@126.com}
,~Wenwen Fan\thanks{Email:fwwen123@126.com}, Xiaoping Xie \thanks{Corresponding author. Email: xpxie@scu.edu.cn}\\
{School of Mathematics, Sichuan University, Chengdu
610064, China}
 }

\date{}
\maketitle
\begin{abstract}
%This paper considers  finite element analysis of plane elasticity equations with stochastic Young's modulus and stochastic loads. Firstly,  we apply Karhunen-Lo$\grave{e}$ve expansion  or
%gPC (generalized polynomial chaos) expansion to
%   express stochastic  Young's modulus and stochastic  loads,  and turn the original problem into a  new one
%   containing  a  finite number of deterministic parameters.
% Then we deal with the stochastic part by polynomial approximation (gPC approximation and any polynomial approximation) and the space part by  hybrid stress method.  Uniform priori error estimates are derived in the sense that the error bound is independent of the relevant Lam$\acute{e}$ constant $\lambda$. Finally,  we verify the theoretical results by some numerical experiments.
This paper considers  stochastic hybrid stress quadrilateral   finite element analysis of plane elasticity equations with stochastic Young's modulus and stochastic loads. Firstly,  we apply Karhunen-Lo$\grave{e}$ve expansion   to
    stochastic  Young's modulus and stochastic  loads so as to turn the original problem into a  system
   containing  a  finite number of deterministic parameters.
 Then we  deal with the stochastic field and the space field  by   $k-$version/$p-$version finite element methods   and  a hybrid stress quadrilateral finite element method, respectively.    We show that  the derived a priori error estimates are uniform with respect to the Lam$\acute{e}$ constant $\lambda\in (0, +\infty)$. Finally,  we provide some numerical results.

%This paper considers  finite element analysis of plane elasticity equations with stochastic Young's modulus and stochastic loads. Firstly,  we apply Karhunen-Lo$\grave{e}$ve expansion  or
%gPC (generalized polynomial chaos) expansion to
%    stochastic  Young's modulus and stochastic  loads so as to turn the original problem into a  system
%   containing  a  finite number of deterministic parameters.
% Then we  deal with the stochastic domain by a gPC approximation  and the space domain by a hybrid stress quadrilateral finite element method.    Uniform a priori error estimates are derived in the sense that the error bound is independent of the relevant Lam$\acute{e}$ constant $\lambda$. Finally,  we verify the theoretical results by some numerical experiments.
% \vskip 0.4cm {\bf Keywords.}
% stochastic plane elasticity   ~~~ hybrid stress finite element   ~~~   generalized polynomial chaos ~~~  uniform error estimate
\end{abstract}

\renewcommand{\theequation}{\thesection.\arabic{equation}}

%%%%%%%%%%%%%%%%%%%%%%%%%%%%%%%%%%%%%%%%%%%%%%%%%%%%%%%%%%%%%%%%%%%%%

\section{Introduction}

    Let $D\subset R^2$ be a bounded, connected,  convex and open set with boundary $ \partial D=\partial D_0\cup \partial D_1$ and meas($\partial D_0$) $>$ 0, and let ($\Omega$,$\mathcal{F}$,$\mathcal{P}$) be a complete probability space, where $\Omega$, $\mathcal{F}$, $\mathcal{P}$ denote respectively the set of outcomes, the $\sigma$-algebra of subsets of $\Omega$ and the probability measure. Consider the following stochastic plane elasticity equations:  for almost  everywhere (a.e.) $\theta\in\Omega$
   %   The stochastic elasticity equation is given by
%\begin{equation}
%\label{problem}
% \left\{
%    \begin{array}{ll}
%       -div \bm{\sigma}=\textbf{f} , &in ~D \\
%      \bm{\sigma} = \mathcal{C}\epsilon(\textbf{u}),  &in~ D \\
%      \textbf{u}|_{\Gamma_d}=0,     &\bm{\sigma} \textbf{n}|_{\Gamma_n}=\textbf{g}
%       \end{array}
%       \right.
%\end{equation}
\begin{equation}
\label{problem}
\left\{
    \begin{array}{ll}
  -{\bf div} {\bm{\sigma}}(\cdot,\theta)= \textbf{f}(\cdot,\theta), &\text{in}~D ,\\
 \bm{\sigma}(\cdot,\theta) = \mathcal{C}\epsilon(\textbf{u}(\cdot,\theta)),  &\text{in}~D, \\
      \textbf{u}(\cdot,\theta)|_{\partial D_0  }=0,
     {\bm{\sigma}}(\cdot,\theta)\textbf{n}|_{\partial D_1  }=\textbf{g}(\cdot,\theta),&
       \end{array}
       \right.
\end{equation}
where $\bm{\sigma}:\overline{D}\times\Omega\rightarrow R_{sym} ^{2\times2}$ denotes the symmetric stress tensor field, $\textbf{u}:\overline{D}\times\Omega\rightarrow R^2$ the displacement field,
   $\epsilon(\textbf{u})=(\bigtriangledown \textbf{u}+\bigtriangledown^T \textbf{u})/2$ the strain  with $\bigtriangledown=(\frac{\partial}{\partial x_1},\frac{\partial}{\partial x_2})^T$ for $\textbf{x}=(x_1,x_2)$,
  % in this paper, the gradient notation $\bigtriangledown$ always means differentiation with respect to $\textbf{\textbf{x}}\in D$ only, unless otherwise stated.
   $\textbf{f}:D\times\Omega\rightarrow R^2$ the body loading density and $\textbf{g}: \partial D_1 \times\Omega\rightarrow R^2$ the surface traction, $\textbf{n}$ the unit outward vector normal to $\partial D$,  $\mathcal{C}$ the elasticity modulus tensor with
$$\mathcal{C}\epsilon(\textbf{u})=2\mu\epsilon(\textbf{u})+\lambda \mbox{div}\textbf{u} \textbf{I},$$
 \textbf{I} the $2\times2$ identity tensor, and $\mu,\lambda$ the Lam$\acute{e}$ parameters given by $\mu=\frac{\widetilde{E}}{2(1+\nu)}$, $\lambda=\frac{\widetilde{E}\nu}{(1+\nu)(1-2\nu)}$ for plane strain problems and by $\mu=\frac{\widetilde{E}}{2(1+\nu)}$, $\lambda=\frac{\widetilde{E}}{(1+\nu)(1-\nu)}$ for plane stress problems,  with $\nu\in (0,0.5)$ the Poisson ratio and $\widetilde{E}:D\times \Omega\rightarrow R$ the Young's modulus which is stochastic with
\begin{equation}
\label{E assumption}
 0<e_{min}\leq \widetilde{E}(\textbf{\textbf{x}},\theta)\leq e_{max} ~~~\text{ a.e. in } D\times\Omega
 \end{equation}
 for positve constants  $e_{min}$ and $e_{max}$.  Since in the analysis of this paper we need to use an explicit form of $\widetilde{E}$, we rewrite  the second equation of \eqref{problem} as
\begin{equation}\label{new C}
\bm{\sigma}(\cdot,\theta) = \widetilde{E}{\textbf{C}}\epsilon(\textbf{u}(\cdot,\theta)),
\end{equation}
where  the tensor $ {\textbf{C}}:=\frac{1}{\widetilde{E}}  \mathcal{C}$ depends only on  the Poisson ratio $\nu$.

%In classical  computations, or  when  loads $\textbf{f}$, $\textbf{g}$ and Young's modulus $\widetilde{E}$ in  (\ref{problem}) are deterministic, the stochastic plane elasticity equations coincide with the deterministic ones.
 It is well-known that the standard 4-node displacement quadrilateral element (abbr.  bilinear element)  yields poor results for deterministic plane elasticity equations with bending and, for deterministic plane strain problems, at
the nearly incompressible limit.  To improve its performance,   Wilson et al.  \cite{Wilson element,improved Wilson element} developed methods of incompatible modes by enriching the standard (compatible) displacement modes with internal incompatible displacements.    Pian and Sumihara  \cite{19 yu} proposed a hybrid stress  quadrilateral element (PS   element)   based on Hellinger-Reissner variational principle,
  where the displacement vector is approximated by  isoparametric bilinear  interpolations,
   and the stress tensor by a piecewise-independent 5-parameter mode.  Xie and Zhou \cite{34 yu,35 yu} derived robust 4-node hybrid stress quadrilateral elements by optimizing stress modes with a so-called energy-compatibility condition, i.e. the assumed stress terms are orthogonal to the enhanced strains caused by Wilson bubble displacements.
      In \cite{yu 39}  Zhou and Xie gave  a  unified analysis for some hybrid stress/strain quadrilateral methods,
 but the upper bound in the error estimate is not uniform with respect to the Lam$\acute{e}$ parameter $\lambda$.
Yu, Xie and Carstensen \cite{yu guo zhu} derived   uniform convergence results for the hybrid stress methods in \cite{19 yu}  and \cite{34 yu},  in the sense that the error bound is independent of  $\lambda$ .

In the numerical analysis of   stochastic partial differential equations, stochastic finite element methods, which employ  finite elements in the space domain, have gained much attention in the past two decades.   In the probability domain, the stochastic finite element methods use
 two types of  approximation methods,  statistical approximation  and non-statistical approximation.  Monte Carlo sampling(MCs) is one of the most commonly used statistical approximation methods  \cite{Monte Carlo}.  In MCs, one generates realizations of stochastic terms so as to make the problem  deterministic, and only  needs to compute the deterministic problem repeatedly,  and  collect an ensemble of solutions, through which statistical information, such as mean and variance, can be obtained. The disadvantage of MCs lies in the need of a large amount of calculations and its low convergence rate.  There are also some  variants of MCs such as quasi Monte Carlo\cite{H.Niederreiter MC and Quasi-MC} and the stochastic collocation method\cite{collocation, D.Xiu 2012, sparse grid 1, sparse grid 2}.

Non-statistical approximation  methods mainly contain   perturbation methods, Neumann series expansion methods\cite{book stochastic} and so on  at the beginning. But these methods are limited to the magnitude of uncertainties of stochastic terms and the accuracy of calculation.  Later, polynomial approximation is used for the stochastic part.  For example,
 Polynomial chaos (PC) expansion is applied   in  \cite{wiener propose chaos, book stochastic}  to represent  solutions
formally and obtain  solutions by solving  the expansion coefficients \cite{1990 R.G,galerkin methods linear and nonlinear}.  Generalized polynomial chaos (gPC) is  used to express  solutions in   \cite{Lucor.D generalized pc, xiu dong bin 2009 review,D.xiu gpc solve problem}.  According to \cite{xiu dong bin wiener-askey},  one can achieve  exponential  convergence when optimum gPC is  chosen.  Subsequently, it was further generalized \cite{babuska 2004,duan stochastic} that   p version, k version and p-k-version finite element methods could be used  for the approximation of the stochastic part.
 %any set of complete basis, such as total degree polynomial basis and piecewise
 %polynomial basis
 % can be used for the approximation of stochastic part.
%Babu$\check{s}$ka used $p\times h$ version and $k\times h$ version finite %element methods to solve stochastic partial differential equations.

So far,  there are very  limited  studies  on the numerical solution of  the stochastic plane elasticity equations (\ref{problem}). In  \cite{perturb} a generalized $n$th order stochastic perturbation technique is implemented in conjunction with  linear finite elements to model a 1D linear elastostatic problem with a single random variable.
In \cite{1990 R.G} the numerical solution of problem (\ref{problem}) is considered  with stochastic  Young's modulus $\widetilde{E}$, where PC approximation  and bilinear finite elements are  applied respectively to the stochastic domain and the space domain. We refer to \cite{stochastic elasticity,stochastic elasticity 2} for some other related studies.
%As to solve the stochastic  equation, there are  many methods too, such as early Monte carlo simulation, perturbation method, neumann expansion method in \cite{book stochastic} and so on, also in  \cite{babuska 2004} Babu$\check{s}$ka  proposes stochastic Galerkin finite element methods h-version,k-version, $p\times h$,$k\times h$ and gives corresponding error estimates. In \cite{collocation},stochastic collocation method with convergence are discussed.  There have many other research efforts in \cite{duan stochastic}\cite{xiu dong bin 2009 review}\cite{xiu dong bin wiener-askey}\cite{galerkin methods linear and nonlinear}. In \cite{mix stochastic}, mixed finite element method is used to solve stochastic pressure equation of Wick type, but for the stochastic elasticity equation, so far, almost no one have used  mixed finite element method  to solve it and obtain the  uniform error analysis with respect to the nearly incompressibility for the stochastic  assumed stress hybrid methods on arbitrary quadrilateral meshes.
 In this contribution, we shall propose and analyze  stochastic $k\times h-$version and $p\times h-$version finite element methods  for  the problem (\ref{problem}), where
we use  $k-$version/$p-$version finite element methods   for the stochastic domain and PS hybrid stress quadrilateral finite element for the space domain.

 We arrange the   paper as follows. In Section 2 we show stochastic mixed variational formulations of (1.1),  and give the    existence and uniqueness of the weak solution.
 Section 3 discusses the  approximation of  the stochastic coefficient and  stochastic loads, as well as the truncated stochastic mixed variational formulations.  Section 4 analyzes the proposed stochastic $k\times h-$version and $p\times h-$version
 finite element methods
    and derives uniform a priori error estimates.  Finally,   Section 5 provides   some numerical results.

%%%%%%%%%%%%%%%%%%%%%%%%%%%%%%%%%%%%%%%%%%%%%%%%%%%%%%%%%%%%%%%%%%%%%%%%%%%%%%%%%%%%%%%%%%%%%
%%%%%%%%%%%%%%%%%%%%%%%%%%%%%%%%µ⁄∂˛Ω⁄%%%%%%%%%%%%%%%%%%%%%%%%%%%%%%%%%%%%%%%%%%%%%%%%%%%%%%%
%%%%%%%%%%%%%%%%%%%%%%%%%%%%%%%%%%%%%%%%%%%%%%%%%%%%%%%%%%%%%%%%%%%%%%%%%%%%%%%%%%%%%%%%%%%%%

%\section{Theoretical aspect of the continuous problem}
\section{Stochastic mixed variational formulations}
\subsection{Notations} For the probability space $ (\Omega, \mathcal{F},\mathcal{P})$ and an  integer $m$, denote
%Let $Y$ be a stochastic variable in ($\Omega$,$\mathcal{F}$,$\mathcal{P}$).
%Denote by $L^2_P(\Omega)$ the set of all  random variables $Y$ with
$$L^m_P(\Omega):=\left\{Y| \ Y \text{ is a random variable in }  (\Omega, \mathcal{F},\mathcal{P}) \text{ with } \int_{\Omega}|Y(\theta)|^m \mathrm{d}P(\theta)< +\infty\right\}.$$
 If  $Y\in L_P^1(\Omega)$, we denote its expected value by
\begin{equation}
\label{expection}
   E[Y]=\int_{\Omega}Y(\theta) \mathrm{d}P(\theta)=\int_{R}y\mathrm{d}F(y),
\end{equation}
where $F$ is the distribution probability measure of $Y$,  given by $F(B)=P(Y^{-1}(B))$ for any borel set $B$ in $R$.  Assume that $F(B)$ is absolutely continuous with respect to Lebesgue measure, then there exists a density function for $Y$, $\rho: R\rightarrow [0,+\infty)$, such that
\begin{equation}
E[Y]=\int_R y \rho(y) \mathrm{d}y.
\end{equation}
%Let $L^2(D; X)$ be the space of square integrable functions defined on $D$
%with values in the finite-dimensional vector space $X$ and with norm being denoted
%by $||\cdot||_{0}$ and $L^{\infty}(D;X)$.

We denote by  $H^m(D)$  the usual Sobolev space consisting of functions defined on the domain $D$,  with all derivatives of order
up to $m$ square-integrable. Let  $(\cdot,\cdot)_{H^m(D)}$% $(\cdot,\cdot)_{H^m(K; X)}$
be the usual inner product on $H^m(D)$. The norm $||\cdot||_{m}$ on $H^m(D)$ deduced by $(\cdot,\cdot)_{H^m(D)}$  is given  by
$$||v||_{m}:=(\sum_{0\leq j \leq m}|v|_{j}^2)^{1/2} \text{ with the semi-norm }
|v|_{j}:=(\sum_{|\alpha|=j}||D^{\alpha}v||_{0}^2)^{1/2}.$$
 In particular,
$L^2(D):=H^0(D)$. Denote
$$L^{\infty}(D):=\{ w :   \ ||w||_{\infty}:=ess sup_{x\in D}|w(x)|<\infty\}.$$
 We define the following stochastic Sobolev spaces:
$$L^2_P(\Omega;H^m(D)):=\{w :%D\times \Omega\rightarrow X ~|
~ w  \text{ is  strongly  measurable  with } w(\cdot,\theta) \in H^m(D) \text{ for $\theta\in \Omega$ and } ||w||_{\widetilde{m}} <+\infty\},$$
$$L^{\infty}_P(\Omega;L^{\infty}(D)):=\{w :%D\times\Omega\rightarrow X ~|
~ w   \text{ is  strongly  measurable  with }  w(\cdot,\theta) \in L^\infty(D) \text{ for $\theta\in \Omega$ and }   ||w||_{\widetilde{\infty}}<+\infty \},$$
where the norms $||\cdot||_{\widetilde{m}}$, $||\cdot||_{\widetilde{\infty}}$ are respectively defined as
\begin{equation}\label{norm-m-tilde}
||w||_{\widetilde{m}}:=(E[||w( \cdot,\theta )||^2_{m}])^{\frac{1}{2}},\quad  ||w||_{\widetilde{\infty}}:=ess sup_{\theta \in \Omega}||w(\cdot ,\theta )||_{\infty}.
\end{equation}

On the other hand,
since stochastic functions intrinsically have different structures with respect to $\theta \in \Omega$ and $\textbf{x}\in D$, we follow \cite{babuska 2004} to introduce tensor spaces for the analysis of numerical approximation. Let $X_1 (\Omega)$, $X_2(D)$ be Hilbert spaces. The tensor spaces $X_1(\Omega)\otimes X_2(D) $ is the completion of formal sums
   $\phi (\theta,\textbf{x})=\Sigma_{i=1,...,n}u_i(\theta)v_i(\textbf{x}),u_i \in X_1(\Omega),v_i\in X_2(D)$, with respect to the inner product$(\phi,\widehat{\phi})_{X_1\otimes X_2}:=\Sigma_{i,j}(u_i,\widehat{u_j})_{X_1}(v_i,\widehat{v_j})_{X_2}$.
  % Thus, if $\phi(\theta,\textbf{x})\in X_1(\Omega)\otimes X_2(D)$, then $\phi(\theta,  )\in X_2(D)$ a.e. on $\Omega$, and $\phi( ,\textbf{x})\in L^2(\Omega)$ a.e. on D.
  % We have the following  isomorphisms for the tensor spaces $L^2(\Omega)\otimes H^s(D)^2$ and
Then, for the tensor space $L^2_P(\Omega) \otimes H^m(D)$,
%with tensor inner product
%$$(w,\widehat{w})_{L^2_P(\Omega) \otimes H^m(D; X)}=\int_{\Omega}(w(\cdot, \theta),\widehat{w}(\cdot, \theta))_{H^m(D; X)}\mathrm{d}P(\theta).
%$$
 we have the  following  isomorphism:
$$ L^2_P(\Omega;H^m(D )) \simeq L^2_P(\Omega) \otimes H^m(D ).$$

    For convenience, we use the notation $a\lesssim b$ to represent that there exists a generic positive constant C such that  $a\leq Cb$, where $C$ is independent of the Lam$\acute{e}$ constant $\lambda $ and the  mesh parameters $h$, $ k$,   the polynomial degree $p$ in the stochastic $k\times h-$version and $p\times h-$version
 finite element methods.

\subsection{Weak formulations}
 Introduce the spaces
 %:=H^1_d(D)
$${V_D}:=\{v\in H^1(D)^2: v|_{\partial D_0}=0\},$$
$$\small {\Sigma_D}:=
\left\{
    \begin{array}{ll}
       L^2(D;R^{2\times 2}_{sym}):=\{\tau:D\rightarrow R^{2\times 2} |\ \tau_{ij}\in L^2(D),\ \tau_{ij}=\tau_{ji},\ i,j=1,2\}, & \text{if}~~ \text{meas}(\partial D_1)>0,\\
       \{\bm\tau\in L^2(D;R^{2\times 2}_{sym}): \int_D tr\bm\tau\mathrm{d}\textbf{x}=0 \text{ with trace } tr\bm{\tau}: =\bm{\tau}_{11}+\bm{\tau}_{22}\},&\text{if}~~\partial D_1=\emptyset.
       \end{array}
       \right.$$
 % $$V:=L^2_P(\Omega ;  ~{V_D}), \quad \quad\quad\quad\quad \Sigma:=L^2_P(\Omega;~ \Sigma_D).$$
%$$\Sigma:=
%\left\{
%    \begin{array}{ll}
%       L^2(\Omega;L^2(D;R^{2\times 2}_{sym}), & if~~ meas(\Gamma_n)>0,\\
%       \{\bm\tau\in L^2(\Omega;L^2(D;R^{2\times 2}_{sym}): \int_D tr\bm\tau\mathrm{d}\textbf{x}=0~~ a.e.~ on ~ \Omega\},&if~~\Gamma_n=\emptyset,
%       \end{array}
%       \right.$$
 %where  $tr\bm{\tau}: =\bm{\tau}_{11}+\bm{\tau}_{22}$ represents the trace of   tensor $\bm{\tau}$.
 Then the weak problem for the model (\ref{problem}) reads as: Find $(\bm{\sigma},\textbf{u})\in L^2_P(\Omega;~ \Sigma_D)\times L^2_P(\Omega ;  ~{V_D}) $ such that
\begin{equation}
\label{continous formulation}
\left\{
    \begin{array}{ll}
        a(\bm{\sigma},\bm{\tau})-b(\bm{\tau}, \textbf{u})=0,        &\forall \bm{\tau}\in L^2_P(\Omega;~ \Sigma_D),  \\
        b(\bm{\sigma}, \textbf{v})=\ell(\textbf{v}),    &\forall \textbf{v}\in L^2_P(\Omega ;  {V_D}),
       \end{array}
       \right.
       \end{equation}
where
 the  bilinear forms     $a(\cdot,\cdot):L^2_P(\Omega;~ \Sigma_D) \times L^2_P(\Omega;~ \Sigma_D) \rightarrow R $, $b(\cdot,\cdot):L^2_P(\Omega;~ \Sigma_D) \times L^2_P(\Omega ;  ~{V_D}) \rightarrow R$ and the linear form $\ell:L^2_P(\Omega ;  ~{V_D})\rightarrow R$  are defined respectively by
\begin{equation}
\label{definition of a( , )}
a(\bm{\sigma},\bm{\tau}):=E[\int_{D}\frac{1}{\widetilde{E}} \bm{\sigma} : {\textbf{C}}^{-1}\bm{\tau}\mathrm{d}\textbf{x}]
=\int_{\Omega}\int_D\frac{1}{\widetilde{E}} \bm{\sigma} : {\textbf{C}}^{-1}\bm{\tau}\mathrm{d}\textbf{x}\mathrm{d}P(\theta) ,%\frac{1}{2\mu} \bm{\sigma}^D:\bm{\tau}^D+\frac{1}{4(\mu+\lambda)}tr\bm{\sigma}tr\bm{\tau} \mathrm{d}\textbf{x} \mathrm{d}P(\theta),
\end{equation}
\begin{equation}
b(\bm{\tau},\textbf{u}):=E[\int_D \bm{\tau} : \epsilon (\textbf{u})\mathrm{d}\textbf{x}]=\int_\Omega\int_D \bm{\tau} : \epsilon(\textbf{u})\mathrm{d}\textbf{x} \mathrm{d}P(\theta),\end{equation}
\begin{equation}
\label{definition of l( , )}
\ell(\textbf{v}):=E[\int_D \textbf{f}\textbf{v}\mathrm{d}\textbf{x}+\int_{\partial D_1} \textbf{g}\cdot\textbf{v}\mathrm{d}s]=\int_\Omega \int_D \textbf{f}\textbf{v}\mathrm{d}\textbf{x} \mathrm{d} P(\theta)+\int_\Omega\int_{\partial D_1} \textbf{g}\cdot\textbf{v}\mathrm{d}s \mathrm{d}P(\theta).
\end{equation}
Here $\bm{\sigma}:\bm{\tau}=\sum_{i,j=1}^2 \bm{\sigma}_{ij} \bm{\tau}_{ij}$.%${\bm{\tau}}^D:={\bm{\tau}}-\frac{1}{2}tr {\bm{\tau}} I$.

%\subsection{Existence and uniqueness for the solution of the stochastic elasticity problem}
It is easy to see that the following continuity conditions hold: for $\bm{\sigma},\bm{\tau} \in L^2_P(\Omega;~ \Sigma_D)$, $ \textbf{v}\in L^2_P(\Omega ;  ~{V_D})$,
\begin{equation}
a(\bm{\sigma},\bm{\tau})\lesssim ||\bm{\sigma}||_{\widetilde{0}}~||\bm{\tau}||_{\widetilde{0}},\quad
b(\bm{\tau},\textbf{v})\lesssim ||\bm{\tau}||_{\widetilde{0}}~|\textbf{v}|_{\widetilde{1}} ,\quad \ell(\textbf{v})\lesssim (||\textbf{f}||_{\widetilde{0}}+||\textbf{g}||_{\widetilde{0},\partial D_1})~|\textbf{v}|_{\widetilde{1}} .
\end{equation}
According to the theory of mixed finite element methods \cite{book hybrid}\cite{F.B}, we need the following two stability conditions for the well-posedness of the weak problem (\ref{continous formulation}):

(\textbf{A}) Kernel-coercivity: for any $\bm{\tau} \in Z^0:=\{\bm{\tau} \in L^2_P(\Omega;~ \Sigma_D): b(\bm{\tau},\textbf{v}) =0,~\forall~ \textbf{v}\in L^2_P(\Omega ;  ~{V_D})\}$
it holds
\begin{equation}
\label{A}
||\bm{\tau}||^2_{\widetilde{0}}\lesssim a(\bm{\tau},\bm{\tau)}.
\end{equation}

(\textbf{B}) Inf-sup condition: for any $\textbf{v}\in L^2_P(\Omega ;  ~{V_D})$ it holds
\begin{equation}
\label{B}
|\textbf{v}|_{\widetilde{1}}\lesssim \sup_{0\neq \bm{\tau} \in L^2_P(\Omega;~ \Sigma_D)} \frac{b(\bm{\tau},\textbf{v})}{||\bm{\tau}||_{\widetilde{0}}}.
\end{equation}
\begin{thm}
The uniform stability conditions $(\textbf{A})$ and $(\textbf{B})$ hold.
\end{thm}
\begin{proof}
For any $\bm{\tau}\in Z^{0}$, we have,  a.e. $\theta \in \Omega$,  $\bm{\tau}(\cdot, \theta)\in \{\bm{\tau} \in {\Sigma_D}: \int_D \bm{\tau} : \epsilon (\textbf{v})\mathrm{d}\textbf{x} =0~~~\forall~ \textbf{v}\in {V_D}\}$.
According to Theorem 2.1 in \cite{yu guo zhu} and the assumption \eqref{E assumption}, it holds
$$\int_D\bm{\tau}(\cdot,\theta):\bm{\tau}(\cdot,\theta)\mathrm{d}\textbf{x}\lesssim \int_D \frac{1}{\widetilde{E}} \bm{\tau}(\cdot,\theta):\textbf{C}^{-1}(\cdot,\theta)\bm{\tau}(\cdot,\theta)\mathrm{d}\textbf{x},$$
which leads to $$\int_{\Omega}\int_D\bm{\tau}:\bm{\tau}\mathrm{d}\textbf{x}\mathrm{d}P(\theta)\lesssim \int_{\Omega}\int_D\frac{1}{\widetilde{E}}\cdot \bm{\tau}:\textbf{ C } ^{-1}\bm{\tau}\mathrm{d}\textbf{x}\mathrm{d}P(\theta),$$
i.e. $(\textbf{A})$ holds.

Let $\textbf{v}\in L^2_P(\Omega ;  ~{V_D})$ and notice $\epsilon(\textbf{v})\in L^2_P(\Omega;~ \Sigma_D)$. Then
 $$|\epsilon(\textbf{v})|_{\widetilde{0}}\leq \sup_{\bm{\tau}\in L^2_P(\Omega;~ \Sigma_D) \backslash \{0\}}\frac{\int_{\Omega}\int_D\bm{\tau}:\epsilon(\textbf{v})\mathrm{d}\textbf{x}\mathrm{d}P(\theta)}{||\bm{\tau}||_{\widetilde{0}}}.$$
 Hence $(\textbf{B})$ follows from the equivalence between the two norms $|\epsilon(\textbf{v})|_{\widetilde{0}}$ and $|\textbf{v}|_{\widetilde{1}}$ on $L^2_P(\Omega ;  ~{V_D})$.
\end{proof}

In view of the above conditions, we immediately obtain the following well-posedness result:
\begin{thm}
Assume that $\textbf{f}\in L_P^2(\Omega, L^2(D)^2)$, $\textbf{g}\in L_P^2(\Omega,L^{2}(\partial D_1)^2)$. Then the weak problem (\ref{continous formulation}) admits a unique solution $(\bm{\sigma} , \textbf{u})\in L^2_P(\Omega;~ \Sigma_D) \times L^2_P(\Omega ;  ~{V_D})$ such that
\begin{equation}
\label{stability}
||\bm{\sigma}||_{\widetilde{0}}+|\textbf{u}|_{\widetilde{1}}\lesssim ||\textbf{f}||_{\widetilde{0}}+||\textbf{g}||_{\widetilde{0},\partial D_1}.
\end{equation}
\end{thm}
\section{ Truncated stochastic mixed variational formulations}
 In  order to solve the weak problem (\ref{continous formulation})   by deterministic numerical methods, we firstly approximate
the stochastic coefficient  $\widetilde{E}$  and the loads $\textbf{f}$, $\textbf{g}$ by  using a finite number of random variables; we refer to \cite{a state-of-the-art report} for several  approximation approaches. Here, we only consider  the Karhunen-Lo$\grave{e}$ve(K-L) expansion.
\subsection{Karhunen-Lo$\grave{e}$ve(K-L) expansion}
%In  order to   transform (\ref{continous formulation})  into one that's
%possible to solve by deterministic numerical methods, we need to describe  or approximate
%stochastic coefficient  $\widetilde{E}$  and loads $\textbf{f}$, $\textbf{g}$ by a finite number of random variables $[Y_1,Y_2,...,Y_N]: \Omega \rightarrow R^N$,  and  we denote their approximations by $\widetilde{E}_N$  $\textbf{f}_N$, $\textbf{g}_N$.  Several such approaches are available \cite{}. Here, we focus on  Karhunen-Lo$\grave{e}$ve(K-L) expansion.

 For any stochastic process  $\phi(\textbf{x},\theta)\in L_P^2(\Omega;L^2(D))$  with  covariance function
$cov[\phi](\textbf{x}_1,\textbf{x}_2):D \times D \rightarrow R $ , which is bounded, symmetric and positive definitely. Let $\{(\lambda_n,b_n)\}_{n=1}^\infty$ be the sequence of eigenpairs satisfying
\begin{equation}
\label{integral equation cov}
\int_D~ cov~[\phi]~(\textbf{x}_1,\textbf{x}_2)~b_n(\textbf{x}_2)~\mathrm
{d}\textbf{x}_2=\lambda_n b_n(\textbf{x}_1),
\end{equation}
\begin{equation}
 \label{sum of lambda}
 \sum_{n=1}^{+\infty} \lambda_n=\int_{ D}cov[\phi](\textbf{x},\textbf{x})\mathrm{d}\textbf{x}, \quad \int_D~b_i(\textbf{x})b_j(\textbf{x})~\mathrm
{d}\textbf{x}=\delta_{ij}, \ i,j=1,2,\cdots,
 \end{equation}
and $\lambda_1\geq\lambda_2\geq \cdots>0$.
Then the Karhunen-Lo$\grave{e}$ve(K-L) expansion of $\phi(\textbf{x},\theta)$ is given by
\begin{equation}
\label{K-L}
  \phi(\textbf{x},\theta)=E[\phi](\textbf{x})+\sum^{\infty}_{n=1} \sqrt{\lambda_n}b_n(\textbf{x})Y_n(\theta),
\end{equation}
and the truncated K-L expansion of $\phi(\textbf{x},\theta)$ is
\begin{equation}
\label{truncated K-L}
 \phi_N(\textbf{x},\theta)= E[\phi](\textbf{x})+\sum^{N}_{n=1} \sqrt{\lambda_n}b_n(\textbf{x})Y_n(\theta).
\end{equation}
Here  $\{Y_n\}^{\infty}_{n=1}$ are mutually uncorrelated with mean zeros and unit variance with $Y_n(\theta)=\frac{1}{\sqrt \lambda_n} \int_D(\phi(\textbf{x},\theta)-E[\phi](\textbf{x})) b_n(\textbf{x}) \mathrm{d}\textbf{x}$.

%\begin{rem}
%For  the eigenpairs $\{(\lambda_n,b_n)\}_{n=1}^\infty$, it is easy to see the following relations hold:
% \begin{equation}
% \label{sum of lambda}
% \sum_{n=1}^{+\infty} \lambda_n=\int_{ D}cov[\phi](\textbf{x},\textbf{x})\mathrm{d}\textbf{x}, \quad \int_D~b_i(\textbf{x})b_j(\textbf{x})~\mathrm
%{d}\textbf{x}=\delta_{ij}, \ i,j=1,2,\cdots.
% \end{equation}
% \end{rem}

By Mercer's theorem \cite{K-L convergence theorem}, it holds
\begin{equation}
\label{convergence K-L}
\sup_{\textbf{x}\in D} E[(\phi-\phi_N)^2](\textbf{x})=\sup_{\textbf{x}\in D} \sum_{n=N+1}^{+\infty} \lambda_n b_n^2(\textbf{x})\rightarrow 0 . ~~~~~~~~ as~~~  N\rightarrow \infty.
\end{equation}

In what follows  we show the estimation of the  truncated error $\phi-\phi_N$ in norms   $||\cdot||_{\widetilde{0}}$ and $||\cdot||_{\widetilde{\infty}}$, respectively.

 From (\ref{sum of lambda}) it follows
\begin{equation}\label{trunc-error}
||\phi-\phi_N||_{\widetilde{0}}^2=\sum_{n=N+1}^{+\infty} \lambda_{n}\quad \text{and}\quad ||\phi-\phi_N||_{\widetilde{0}} \rightarrow 0 \quad as \ \ N\rightarrow +\infty.
\end{equation}
 Obviously the convergence rate of $||\phi-\phi_N||_{\widetilde{0}}$ is  strongly   depending on the decay rate of the eigenvalues $\lambda_n$, which ultimately depends on the regularity of the covariance function $cov[\phi]$.
Generally, the smoother the covariance is, the faster the eigenvalues decay, which implies the faster  $||\phi-\phi_N||_{\widetilde{0}} $ converges to zero.
 %Now we need the following definition, which are related  to the regularity of  $cov[\phi] $.
% For piecewise analytic or piecewise smooth covariance function $cov[\phi] $ in the Definition \cite{K-L2006} ,  we quote from [??] the following results.
 Now we quote from  \cite{K-L2006} the following definition (Definition \ref{D3.1}, which are related  to the regularity of  $cov[\phi]$) and  lemma (Lemma \ref{lemma eigenvalue convergence}, which  gives  the   decay rate of the eigenvalues $\lambda_n$).

%Now, we introduce the convergence of $\phi(\textbf{x},\theta)- \phi_N(\textbf{x},\theta)$ in norm $||\cdot||_{L^{\infty}(\Omega \times D)}$ and
%$||\cdot||_{\widetilde{0}}$, which will be used in the following subsection.
%  In order to make sure $||\widetilde{E}-\widetilde{E}_{N}||_{\widetilde{\infty}}$ converges  to zero as $N\rightarrow \infty$, according to [?][?],  its  convergence rate is strongly depending on the regularity of the covariance kernel $\text{cov}[\phi]$.
% Then  we quote from [?][?] the following assumption, definition and lemma.
\begin{define}\label{D3.1}\cite{K-L2006}
 The covariance function $ cov[\phi]: D \times D\rightarrow R$ is said to be piecewise analytic/smooth on  $D \times D$ if there exists a finite family $(D_j)_{1\leq j \leq J}\subset R^2$ of open hypercubes such that $\overline{D}\subseteq \cup _{j=1}^{J}\overline{D}_j$, $D_j\cap D_{j'}=\varnothing, ~~ \forall j\neq j'$ and $cov[\phi]|_{D_j\times D_{j'}}$ has an analytic/smooth continuation in a neighbourhood of
$\overline{D}_j\times \overline{D}_{j'}$ for any pair $(j,j')$.
\end{define}

\begin{lem}\cite{K-L2006}
\label{lemma eigenvalue convergence}
If $cov[\phi]$ is piecewise analytic on $D\times D$, then for the eigenvalue sequence $\{\lambda_n \}_{n\geq 1}$, there exist constants $c_1, c_2$  depending only on $cov[\phi]$ such that
\begin{equation}
0\leq \lambda_n \leq c_1 e^{-c_2 n^{1/2}},~~~~~~~~~\forall n\geq 1.
\end{equation}
If $cov[\phi]$ is piecewise smooth on $D\times D$, then for any constant $s>0$ there exists a constant $c_s$ depending only on $cov[\phi]$ and $s$, such that
\begin{equation}
0\leq \lambda_n \leq c_s n^{-s},~~~~~~~~~\forall n\geq 1.
\end{equation}
\end{lem}
By Lemma \ref{lemma eigenvalue convergence}, we immediately have the following convergence results.
\begin{lem}\label{3.2}
If $cov[\phi]$ is piecewise analytic on $D\times D$, then  there exists   constants $c_1, c_2$ depending only on $cov[\phi]$  such that
\begin{equation}
||\phi-\phi_N||_{\widetilde{0}}\leq  \frac{2c_1}{c_2^2}(1+c_2N^{1/2})e^{-c_2 N^{1/2}}, ~~~\forall N \geq 1.
\end{equation}
If $cov[\phi]$ is piecewise smooth on $D\times D$, then for any $s>0$ there exists
$C_s$ depending only on $cov[\phi]$ and $s$, such that
\begin{equation}
||\phi-\phi_N||_{\widetilde{0}}\leq C_s N^{-s},~~~\forall N \geq 1.
\end{equation}
\end{lem}

 To estimate $||\phi-\phi_N||_{\widetilde{\infty}}$,
we make the following assumption:
\begin{assum}
\label{assumption of Y_n uniformly bounded}
The random variables $\{Y_n(\theta)\}_{n=1}^{\infty}$  in the K-L expansion  are independent and  uniformly bounded with
\begin{equation*}
||Y_n(\theta)||_{L^{\infty}(\Omega)} \leq C_Y,~~~~~  \forall n \geq 1,
\end{equation*}
where $C_Y$ is a positive constant.
\end{assum}
%  Then we quote from the following results.
\begin{lem}\cite{K-L2005,K-L2006}
\label{lemma infty norm KL convergence}
Suppose Assumption \ref{assumption of Y_n uniformly bounded} holds.
 If
$\text{cov}[\phi]$ is piecewise analytic  on $D\times D$,
then there exist a constant $c>0$  such that, for any $s>0$, it holds
\begin{equation}
||\phi-\phi_N||_{\widetilde{\infty}}\leq C e^{-c(1/2-s)N^{1/2}}, \forall N \geq 1,
\end{equation}
where $C$ is a   positive constant  depending on $s,c, \text{cov}[\phi] $ and $J$ given in Definition \ref{D3.1}.
If
$\text{cov}[\phi]$ is piecewise smooth on $D\times D$, then for any
$t>0, r>0$, it holds
\begin{equation}
||\phi-\phi_N||_{\widetilde{\infty}}\leq C' N^{1-t(1-r)/2}, \forall N \geq 1,
\end{equation}
where $C'$ is a positive constant  depending on $t, r, \text{cov}[\phi] $ and $J$.
\end{lem}
\begin{rem}
We note that  we need to solve the integral equation
(\ref{integral equation cov}) to obtain  the K-L expansion \eqref{K-L}.  For some special covariance functions, the equation can be solved analytically \cite{book stochastic},  but for more general cases    numerical methods are required \cite{K-L2005,implementation of K-L2002,K-L2006}. % and we don't dwell on this.

\end{rem}
\subsection{Finite dimensional  approximations of $\widetilde{E}$, $\textbf{f}$, $\textbf{g}$}
In this section, we use  the K-L expansion to approximate $\widetilde{E}$, $\textbf{f}$ and  $\textbf{g}$.

For  $\widetilde{E}$, assume its  truncated  K-L expansion is of the form
\begin{equation}
\label{truncated K-L}
 \widetilde{E}_N(\textbf{x},\theta)=\widetilde{E}_N(\textbf{x},Y_1(\theta),...,Y_N(\theta))= E[\widetilde{E}](\textbf{x})+\sum^{N}_{n=1} \sqrt{\widetilde{\lambda}_n}\widetilde{b}_n(\textbf{x})Y_n(\theta),
\end{equation}
where $\{(\widetilde{\lambda}_n, \widetilde{b}_n(\textbf{x}))\}_{n=1}^{N}$
 and $\{Y_n(\theta)\}_{n=1}^N$ are the corresponding   eigenpairs and random variables, respectively.

  As for $\textbf{f}=(f_1, f_2)^T$ and  $\textbf{g}=(g_1, g_2)^T$, we need to apply the K-L expansion to  each of their components.
%  to the $\textbf{f}$ and  $\textbf{g}$,
% noting  that they both  have two components, namely, $\textbf{f}=(f_1, f_2)^T, \textbf{g}=(g_1, g_2)^T$, which implies that we are  in fact applying
 In this paper,  following  similar  ways as in \cite{babuska 2004,collocation}   to avoid use of more notations, we assume  the truncated K-L expansions   of $\textbf{f}$ and  $\textbf{g}$ take the following  forms:
 \begin{equation}\small
\label{truncated K-L of f}
\textbf{f}_{N}(\textbf{x},\theta)=\textbf{f}_{N}(\textbf{x},Y_1(\theta),...,Y_N(\theta))=\left( \begin{array}{c} f_{1N} \\ f_{2N}\end{array}\right)=\left( \begin{array}{c} E[f_1](\textbf{x}) \\ E[f_2](\textbf{x})\end{array}\right)+\sum^{N}_{n=1} \left( \begin{array}{c} \sqrt{\widehat{\lambda}_{1n}}\widehat{b}_{1n}(\textbf{x}) \\ \sqrt{\widehat{\lambda}_{2n}}\widehat{b}_{2n}(\textbf{x})\end{array}\right)Y_n(\theta),
\end{equation}
 \begin{equation}\small
\label{truncated K-L of g}
\textbf{g}_{N}(\textbf{x},\theta)=\textbf{g}_{N}(\textbf{x},Y_1(\theta),...,Y_N(\theta))= \left( \begin{array}{c} g_{1N} \\ g_{2N}\end{array}\right)=\left( \begin{array}{c} E[g_1](\textbf{x}) \\ E[g_2](\textbf{x})\end{array}\right)+\sum^{N}_{n=1} \left( \begin{array}{c} \sqrt{\overline{\lambda}_{1n}} \overline{b} _{1n} (\textbf{x})\\ \sqrt{\overline{\lambda}_{2n}}\overline{b}_{2n}(\textbf{x})\end{array}\right)Y_n(\theta),
\end{equation}
where $\{(\widehat{\lambda}_{in}, \widehat{b}_{in}(\textbf{x}))\}_{n=1}^{N}$, $\{(\overline{\lambda}_{in}, \overline{b}_{in}(\textbf{x}))\}_{n=1}^{N}$,$i=1,2$
are the corresponding  eigenpairs.
%Then  $\widetilde{E}_N(\textbf{x},\theta)= \widetilde{E}_N(\textbf{x},Y_1(\theta),...,Y_N(\theta)),      \textbf{f}_{N}(\textbf{x},\theta)=\textbf{f}_{N}(\textbf{x},Y_1(\theta),...,Y_N(\theta)),     \textbf{g}_{N}(\textbf{x},\theta)=\textbf{g}_{N}(\textbf{x},Y_1(\theta),...,Y_N(\theta))
%$.
\begin{rem}
\label{remark 1}
In practice,  the Young's modulus $\widetilde{E}$, the body force $\textbf{f}$  and the surface load $\textbf{g}$ may be independent. In such cases, the random variables $\{Y_n(\theta)\}_{n=1}^N$ in the truncated K-L expansions \eqref{truncated K-L}-\eqref{truncated K-L of g} for $\widetilde{E}, f_1,f_2, g_1, g_2$ may be different from each other.
%$
%\widetilde{E}_N(\textbf{x},Y(\theta))=\widetilde{E}_N(\textbf{x},Y_{\widetilde{E}}(\theta))$,
%$\textbf{f}_{N}(\textbf{x},Y(\theta))=\textbf{f}_{N}(\textbf{x},Y_{f}(\theta))$ and
%$\textbf{g}_{N}(\textbf{x},Y(\theta))=\textbf{g}_{N}(\textbf{x},Y_{g}(\theta))$
% with $Y=[Y_{\widetilde{E}},Y_{f},Y_{g}]$ and the vectors $ Y_{\widetilde{E}},Y_{f},Y_{g}$ independent.
However, the analysis of this paper still applies to these cases.
\end{rem}

%\begin{rem}
%\label{remark 3.1}
% When $\phi(\textbf{x},\theta)$ is a Gaussian process, $\{Y_n(\theta)\}_{n=1}^{\infty}$ becomes a set of uncorrelated Gaussian random variables.  $Y_1,Y_2,\cdots,Y_n,\cdots$ are also independent, since the uncorrelation implies the independence for Gaussian random variables.
%\end{rem}
%\begin{rem}
% We note that   K-L expansion is suited and efficient to  represent a stochastic process only when its covariance function is known a priori.
%\end{rem}
\subsection{Truncated mixed formulations }%and its continuity  with respect to  $\widetilde{E}$, $\textbf{f}$, $\textbf{g}$
By replacing ${\widetilde{E}}, \textbf{f}, \textbf{g}$ with their truncated forms
${\widetilde{E}_{N}}, \textbf{f}_N, \textbf{g}_N$  in the bilinear form $a(\cdot,\cdot)$, given in \eqref{definition of a( , )},  and the linear form $\ell(\cdot)$, given in \eqref{definition of l( , )},
we can obtain  the following modified mixed variational formulations for the weak problem \eqref{continous formulation}:
find $(\bm{\sigma}_N,\textbf{u}_N)\in L^2_P(\Omega;~ \Sigma_D)\times L^2_P(\Omega ;  ~{V_D})$ such that
\begin{equation}
\label{truncated continous formulation}
\left\{
    \begin{array}{ll}
        a_N(\bm{\sigma}_N,\bm{\tau})-b(\bm{\tau}, \textbf{u}_N)=0,  &       \forall \bm{\tau}\in L^2_P(\Omega;~ \Sigma_D),  \\
        b(\bm{\sigma}_N, \textbf{v})=\ell_N(\textbf{v}),    &       \forall \textbf{v}\in L^2_P(\Omega ;  ~{V_D}).
       \end{array}
       \right.
       \end{equation}
We recall that $\{Y_n(\theta)\}_{n=1}^N$ are the random variables used in the K-L expansions of $\widetilde{E}$, $\textbf{f}$ and  $\textbf{g}$, which are assumed to satisfy Assumption \ref{assumption of Y_n uniformly bounded}. In what follows we denote
\begin{equation}\label{vector}
Y:=(Y_1, Y_2,...,Y_N),\quad \Gamma_n:=Y_n(\Omega)\subset R, \quad \Gamma:=\prod_{n=1}^{N}\Gamma_n,
\end{equation}
 and let  $\rho: \Gamma \rightarrow R$ be the joint probability density function of random vector $Y$ with
 $\rho \in L^{\infty}(\Gamma)$.
According to  Doob-Dynkin lemma \cite{Doob Dynkin lemma},  the weak solution of the modified problem (\ref{truncated continous formulation})  can be described  by  the  random vector $Y$ as
$$\textbf{u}_N(\textbf{x}, \theta)= \textbf{u}_N(\textbf{x},Y),\quad
\bm{\sigma}_N(\textbf{x},\theta)= \bm{\sigma}_N(\textbf{x},Y),$$
and, by denoting  $\textbf{y}:=(y_1,y_2,\cdots,y_N)$,  the corresponding strong formulation for  (\ref{truncated continous formulation})  is of the form
\begin{equation}
\label{deterministic problem}
\left\{
    \begin{array}{ll}
  -{\bf div} {\bm{\sigma}_N}(\textbf{x},\textbf{y})= \textbf{f}_N(\textbf{x},\textbf{y}), & \forall ( \textbf{x} , \textbf{y} ) \in D \times \Gamma ,\\
 \bm{\sigma}_N(\textbf{x},\textbf{y}) = \widetilde{E}_N \textbf{C}~\epsilon(\textbf{u}_N(\textbf{x},\textbf{y})),  &\forall ( \textbf{x} , \textbf{y} ) \in D \times \Gamma, \\
      \textbf{u}_N(\textbf{x},\textbf{y}) =0,   &\forall ( \textbf{x} , \textbf{y} ) \in \partial D_0 \times \Gamma,\\
     {\bm{\sigma}}_N(\textbf{x},\textbf{y})\textbf{n}   =\textbf{g}_N(\textbf{x},\textbf{y}),&\forall ( \textbf{x} , \textbf{y} ) \in \partial D_1 \times \Gamma.
       \end{array}
       \right.
\end{equation}

Recall that $\rho: \Gamma \rightarrow R$ is the joint probability density function of random vector $Y$. We introduce the    weighted $L^2$-space
\begin{equation}\label{L-rho}
L^2_{\rho}(\Gamma):=\{v: \Gamma\rightarrow R ~|~ \int_{\Gamma}\rho v^2 \mathrm{d}\textbf{y}<+\infty\}.
%,\quad V_{\rho}:=L^2_{\rho}(\Gamma)\otimes {V_D},  \quad\Sigma_{\rho}:=L^2_{\rho}(\Gamma)\otimes {\Sigma_D}.
\end{equation}
We note that from  the norm definition \eqref{norm-m-tilde} it follows
\begin{equation}
||w||_{\widetilde{m}}^2=\int_\Gamma \rho(\textbf{y}) ||w(\cdot,\textbf{y})||_m^2d\textbf{y}= ||w||_{L^2_{\rho}(\Gamma)\otimes {H^m(D)}}^2, \quad \forall  w\in L^2_{\rho}(\Gamma)\otimes {H^m(D)}.
\end{equation}
It is easy to see that  the modified problem (\ref{truncated continous formulation}) is    equivalent to   the following
 deterministic   variational problem:
find $(\bm{\sigma}_N,\textbf{u}_N)\in (L^2_{\rho}(\Gamma)\otimes {\Sigma_D})\times (L^2_{\rho}(\Gamma)\otimes {V_D})$ such that
\begin{equation}
\label{truncated continous formulation N}
\left\{
    \begin{array}{ll}
        a_N(\bm{\sigma}_N,\bm{\tau})-b_N(\bm{\tau}, \textbf{u}_{\rho})=0,&\forall \bm{\tau}\in L^2_{\rho}(\Gamma)\otimes {\Sigma_D},  \\
        b_N(\bm{\sigma}_N, \textbf{v})=\ell_N(\textbf{v}),    &     \forall \textbf{v}\in L^2_{\rho}(\Gamma)\otimes {V_D},
       \end{array}
       \right.
       \end{equation}
where %bilinear forms in (\ref{truncated continous formulation N}) are now rewritten as:
\begin{equation}
\label{a_N definition}
a_N(\bm{\sigma}_N,\bm{\tau}):
=\int_{\Gamma}\rho(\textbf{y})\int_D\frac{1}{\widetilde{E}_{N}} \cdot \bm{\sigma}_N : {{\textbf{C}}}^{-1}\bm{\tau}\mathrm{d}\textbf{x} \mathrm{d}\textbf{y},
\end{equation}
\begin{equation}
\label{b_N definition}
b_N(\bm{\tau},\textbf{u}_N):=\int_\Gamma\rho(\textbf{y})\int_D \bm{\tau} : \epsilon(\textbf{u}_N)\mathrm{d}\textbf{x} \mathrm{d}\textbf{y},
\end{equation}
\begin{equation}
\label{l_N  definition}
\ell_N(\textbf{v}):=\int_\Gamma\rho(\textbf{y})\int_D \textbf{f}_N\textbf{v}\mathrm{d}\textbf{x} \mathrm{d}\textbf{y} +\int_\Gamma\rho(\textbf{y})\int_{\partial D_1} \textbf{g}_N\cdot\textbf{v}\mathrm{d}s \mathrm{d}\textbf{y}.
\end{equation}

The significance of the form (\ref{truncated continous formulation N}) lies in that it  turns the original formulation (\ref{continous formulation}) into a deterministic one with  perturbations of the Young's modulus $\widetilde{E}$, the body force $\textbf{f}$  and the surface load $\textbf{g}$.  Lemma  \ref{propo truncated N} shows, if  the perturbations or the truncated errors   are small enough,   we can    numerically solve the deterministic problem (\ref{truncated continous formulation N}) so as to obtain an approximate solution of the original problem (\ref{continous formulation}).

%\subsection{Truncation of the outcomes set $\Gamma$}
\begin{rem}
In some applications  it may be more efficient to numerically solve the problem
(\ref{truncated continous formulation N}) just in a subdomain  $\widehat{\Gamma} \subset \Gamma$, as, of course, will cause that the corresponding approximation solution has no value in $\Gamma \setminus {\widehat{\Gamma}}$.
\end{rem}

%   In addition,  for the  solutions  of (\ref{continous formulation}) and (\ref{truncated continous formulation}),  we have the following estimate:

 \begin{lem}
 \label{propo truncated N} Suppose that Assumption \ref{assumption of Y_n uniformly bounded} holds and  the  covariance function, $cov[ \widetilde{E}]$, of $\widetilde{E}$ is piecewise smooth (cf. Definition \ref{D3.1}). Then,  for sufficiently large $N$, the modified weak problem (\ref{truncated continous formulation}), or its equivalent problem \eqref{truncated continous formulation N},  admits a unique solution $(\bm{\sigma}_N, \textbf{u}_N)\in (L^2_{\rho}(\Gamma)\otimes {\Sigma_D}) \times ( L^2_{\rho}(\Gamma)\otimes {V_D})$ such that
 \begin{equation}
\label{truncated N error}
||\bm{\sigma}-\bm{\sigma}_N||_{\widetilde{0}}+|\textbf{u}-\textbf{u}_N|_{\widetilde{1}}\lesssim ||{\widetilde{E}}-{\widetilde{E}_{N}}||_{\widetilde{\infty}}\cdot||\bm{\sigma}||_{\widetilde{0}} +||\textbf{f}-\textbf{f}_{N}||_{\widetilde{0}}+||\textbf{g}-\textbf{g}_{N}||_{\widetilde{0},\partial D_1},
%||\bm{\sigma}-\bm{\sigma}_N||_{\widetilde{0}}+|\textbf{u}-\textbf{u}_N|_{\widetilde{1}}\lesssim ||\frac{1}{\widetilde{E}}-\frac{1}{\widetilde{E}_{N}}||_{\widetilde{\infty}}\cdot||\bm{\sigma}||_{\widetilde{0}} +||\textbf{f}-\textbf{f}_{N}||_{\widetilde{0}}+||\textbf{g}-\textbf{g}_{N}||_{\widetilde{0},\partial D_1},
\end{equation}
where
  $( \bm{\sigma}, \textbf{u})\in L^2_P(\Omega;~ \Sigma_D) \times L^2_P(\Omega ;  ~{V_D})$  is the solution of the weak problem (\ref{continous formulation}).% and (\ref{truncated continous formulation})

Moreover, (i) if the  covariance functions  $cov[\widetilde{E}]$, $cov[\textbf{f}]$ and $cov[\textbf{g}]$
are piecewise analytic, then there exists a constant $r>0$, and a constant $C_r>0$ depending only on $cov[\widetilde{E}]$, $cov[\textbf{f}]$, $cov[\textbf{g}]$ and $r$,
such that
\begin{equation}
\label{6}
||\bm{\sigma}-\bm{\sigma}_N||_{\widetilde{0}}+|\textbf{u}-\textbf{u}_N|_{\widetilde{1}}\lesssim  C_rN^{1/2} e^{- r N^{1/2}}.
\end{equation}
(ii) If $cov[\textbf{f}]$ and $cov[\textbf{g}]$
are piecewise smooth,  then for any $s>0$, there exists
$C_s>0$ depending only on $cov[\widetilde{E}]$, $cov[\textbf{f}]$, $cov[\textbf{g}]$ and $s$, such that
\begin{equation}
\label{7}
||\bm{\sigma}-\bm{\sigma}_N||_{\widetilde{0}}+|\textbf{u}-\textbf{u}_N|_{\widetilde{1}}\lesssim  C_s N^{-s}.
\end{equation}
\end{lem}
 \begin{proof}
We first show the modified problem (\ref{truncated continous formulation}) is well-posed. Since the uniform stability conditions for the bilinear form $b(\cdot,\cdot)$ and the linear form $\ell_N(\cdot)$ hold,     it suffices to show  that $\widetilde{E}_N$  is, for sufficiently large $N$, uniformly bounded  with lower bound away from zero a.e. in $D\times \Omega$. In view of Lemma \ref{lemma infty norm KL convergence} and the assumption \eqref{E assumption},  there exists a positive integer  $N_0$ such that,
  for any $N > N_0$, it holds
 \begin{equation}
 \label{truncated E bounded}
 e_{min}' \leq \widetilde{E}_N \leq e_{max}' ~~~a.e.~~~~ \text{in} ~ D\times \Omega,
 \end{equation}
where   $e_{min}'$ and $e_{max}'$ are two  positive constants depending only on  the bounds of $\widetilde{E}$, i.e. $e_{min}$ and $e_{max}$ in \eqref{E assumption}.
 Thus, the corresponding uniform stability conditions  of the bilinear form $a_N(\cdot,\cdot)$ follow from those  of $a(\cdot,\cdot)$.  As a result,  the weak problem (\ref{truncated continous formulation}) admits a unique solution $(\bm{\sigma}_N, \textbf{u}_N)\in L^2_P(\Omega;~ \Sigma_D) \times L^2_P(\Omega ;  ~{V_D})$ with the   stability result
 \begin{equation}
\label{stability N}
||\bm{\sigma}_N||_{\widetilde{0}}+|\textbf{u}_N|_{\widetilde{1}}\lesssim ||\textbf{f}_N||_{\widetilde{0}}+||\textbf{g}_N||_{\widetilde{0},\partial D_1}
\end{equation}
for $N > N_0$.

Next we turn to derive the estimate \eqref{truncated N error}. Subtracting the corresponding  equations in (\ref{continous formulation}) and (\ref{truncated continous formulation}),
%and
%for given $(\textbf{w}, \bm{\gamma}) \in V\times\Sigma$,
 we have
% and  let $E_{\sigma}=\bm{\sigma}-\bm{\sigma_N}, E_{u}=\textbf{u}-\textbf{u}_N$,   we  find the relations
\begin{equation}
\label{1}
\left\{
    \begin{array}{ll}
    a_N(\bm{\sigma}-\bm{\sigma}_N,\bm{\tau})-b(\bm{\tau},\textbf{u}- \textbf{u}_N )=a_N(\bm{\sigma},\bm{\tau})-a(\bm{\sigma},\bm{\tau}),                        &  \forall \bm{\tau}\in L^2_P(\Omega;~ \Sigma_D),  \\
        b(\bm{\sigma}-\bm{\sigma}_N, \textbf{v})=\ell(\textbf{v})-\ell_N(\textbf{v}),   &     \forall \textbf{v}\in L^2_P(\Omega ;  ~{V_D}).
       \end{array}
       \right.
       \end{equation}
Then the desired estimate \eqref{truncated N error} follows from the corresponding stability conditions.
%
% For fixed  $(\textbf{u}, \bm{\sigma})$, the right hand sides of (\ref{1}) respectively define linear functionals on $\Sigma$ and $V$, which are  continuous with norms bounded by
%  $$||\frac{1}{\widetilde{E}}-\frac{1}{\widetilde{E}_{N}}||_{\widetilde{\infty}}\cdot||\bm{\sigma}||_{\widetilde{0}} ~~~~~~\text{and}~~~~~~~ ||\textbf{f}-\textbf{f}_{N}||_{\widetilde{0}}+||\textbf{g}-\textbf{g}_{N}||_{\widetilde{0},\partial D_1} ,$$
%  Then (\ref{truncated N error}) follows from  the stability of (\ref{truncated continous formulation}), i.e. (\ref{stability N}).

  By Lemmas \ref{3.2}-\ref{lemma infty norm KL convergence} and the  estimate \eqref{truncated N error}, we immediately obtain the estimates (\ref{6})-(\ref{7}).
  \end{proof}

\section{  Stochastic hybrid stress
 finite element methods  }
 In this section, we shall consider two types of stochastic finite element methods for the truncated deterministic variational
problem \eqref{truncated continous formulation N}: $k\times h$ version and $p\times h$ version. We use the  PS hybrid stress quadrilateral finite element  \cite{19 yu} to  discretize the space field and $k-$version/$p-$version finite elements to discretize  the stochastic field.

 For convenience we assume that the spacial field $D$ is a convex polygon and the stochastic filed $\Gamma=\prod_{n=1}^N\Gamma_n$ is bounded (cf. Assumption \ref{assumption of Y_n uniformly bounded}).

 \subsection{Hybrid stress  finite element spaces on the spatial field}
%\subsubsection{Tensor product  finite element spaces on stochastic field: $k$-version and $p$-version}
%In this section, we follow the notation conventions from \cite{babuska 2004}.

%\subsubsection{Hybrid stress finite element spaces on spatial domain }
%We next introduce Hybrid stress finite element spaces on the spatial domain $D$.

 Let $\mathcal{T}_h$ be a partition of $\bar D$ by conventional quadrilaterals with the mesh size $h:=max_{T\in\mathcal{T}_h} h_T$,  where
$h_T$ is   the diameter of quadrilateral  $T\in \mathcal T_h$. Let $A_i(x^{(i)}_1,x^{(i)}_2),1\leq i \leq 4$, be the four vertices of T, and $T_i$   the sub-triangle of T with vertices $A_{i-1} ,A_i,A_{i+1}$ (the index of $A_i$ is modulo 4). We assume that the partition $\mathcal{T}_h$ satisfies the following "shape-regularity" hypothesis : there exist a constant $\zeta>2$ independent of h such that, for all $T\in\mathcal{T}_h$, it holds  ~~~~~~~~~~~~~~~~
 \begin{equation}
 \label{domain regular}
 h_T \leqslant \zeta \rho_T,
 \end{equation}
where
$\rho_T:=min_{1\leq i \leq 4}$ \{diameter  of  circle  inscribed in $ T_i$\}.
\begin{figure}[!h]
\begin{center}
\setlength{\unitlength}{0.7cm}
\begin{picture}(10,6)
\put(0,1.5){\line(1,0){3}}        \put(3,1.5){\line(0,1){3}}
\put(0,1.5){\line(0,1){3}} \put(0,4.5){\line(1,0){3}}
\put(0,1.5){\circle*{0.15}}      \put(0,4.5){\circle*{0.15}}
\put(3,4.5){\circle*{0.15}}       \put(0,1.5){\circle*{0.15}}
\put(3,1.5){\circle*{0.15}} \put(-0.5,0.8){\bf{$\hat{A}_{1}$}}
\put(3,0.8){\bf{$\hat{A}_{2}$}}    \put(3,4.7){\bf{$\hat{A}_{3}$}}
\put(-0.5,4.7){\bf{$\hat{A}_{4}$}}  \put(-1,3){\vector(1,0){5}}
\put(1.5,0.5){\vector(0,1){5}} \put(4.2,2.9){$\widehat{x_1}$}
\put(1.5,5.6){$\widehat{x_2}$}

\put(1.6,1.1){-1}                 \put(1.6,4.6){1}
\put(-0.4,3.1){-1} \put(3.1,3.1){1}

\put(5,3){\vector(1,0){1.5}}      \put(5.5,3.2){$F_T$}

\put(8,2.2){\line(6,1){4}}        \put(8,2.2){\line(1,5){0.4}}
\put(8.4,4.2){\line(2,1){2}} \put(10.4,5.2){\line(2,-3){1.6}}
\put(8,2.2){\circle*{0.15}}      \put(8.4,4.2){\circle*{0.15}}
\put(10.4,5.2){\circle*{0.15}}    \put(11.95,2.85){\circle*{0.15}}
\put(7.5,1.9){\bf{$A_{1}$}} \put(12,2.5){\bf{$A_{2}$}}
\put(10.4,5.4){\bf{$A_{3}$}}       \put(8.0,4.5){\bf{$A_{4}$}}
\put(7,1.7){\vector(1,0){5}}      \put(7.5,0.9){\vector(0,1){4}}
\put(12.1,1.4){$x_1$} \put(7.5,5.0){$x_2$}

\end{picture}
\end{center}
\vspace{-1cm} \caption{The mapping $F_{T}$}\label{rr}
\end{figure}
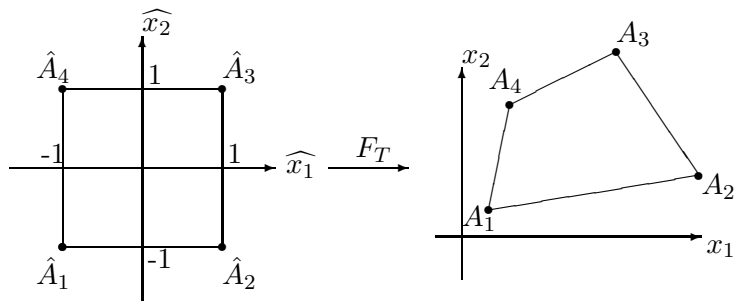

Let $\widehat{T}=[-1,1]\times[-1,1]$ be the reference square with vertices $\widehat{A}_i,1\leq i \leq 4$(Fig.1). Then exists a unique invertible mapping $F_T$ that maps $\widehat{T}$ onto T with  $F_T(\widehat{A}_i)=A_i,1\leq i\leq4$.
%$F_T( \widehat{x}_1, \widehat{x}_2)\in Q^2_1(\widehat{x}_1, \widehat{x}_2)$ and
The isoparametric bilinear mapping $(x_1,x_2)=F_T( \widehat{x}_1, \widehat{x}_2)$ is given by
\begin{equation}
\label{transform}
x_1=a_0+a_1 \widehat{x}_1+a_2 \widehat{x}_1\widehat{x}_2+a_3 \widehat{x}_2 ,~~~~~~~~  x_2=b_0+b_1 \widehat{x}_1+b_2\widehat{x}_1 \widehat{x}_2+b_3\widehat{x}_2 ,
\end{equation}
where   $ \widehat{x}_1, \widehat{x}_2 \in [-1,1]$ are the local isoparametric  coordinates, and
\begin{displaymath}
\left (\begin{array}{cccc}a_0 & b_0  \\ a_1 &
b_1 \\ a_2 & b_2 \\a_3 & b_3  \end{array} \right ):=\frac{1}{4}\left (
\begin{array}{cccc} 1& 1 & 1 & 1 \\ -1 & 1 & 1 & -1 \\ 1 & -1 & 1 & -1 \\ -1 & -1 & 1 & 1
\end{array}\right ) \left ( \begin{array}{cccc} x_1^{(1)} & x_2^{(1)} \\ x_1^{(2)} & x_2^{(2)} \\ x_1^{(3)} & x_2^{(3)} \\ x_1^{(4)} & x_2^{(4)} \end{array}\right ) .
\end{displaymath}

In Pian-SumiharaÕs hybrid stress finite element (abbr. PS element) method for   deterministic plane elasticity problems,  the piecewise isoparametric bilinear interpolation is used for the displacement approximation , namely the   displacement approximation space ${V_D}_h\subset V_D$ is chosen as
\begin{equation}
\label{Q_h}
{V_D}_h:=\{\textbf{v}\in {V_D} : \widehat{\textbf{v}}=v|_T o F_T \in span\{1,\hat x_1,\hat x_2,\hat x_1\hat x_2\}^2,~~\forall\ T\in \mathcal{T}_h\}.
\end{equation}
In other words ,for $\textbf{v}=(\upsilon , \omega)^T \in V_h$ with nodal values $\textbf{v}(A_i)=(\upsilon_i , \omega_i)^T $ on T, $\widehat{\textbf{v}}$ is of the form
$$\widehat{\textbf{v}}=%\sum^4_{i=1}\left( \begin{array}{c} \upsilon_i \\ \omega_i\end{array}\right) N_i(\widehat{x}_1 ,  \widehat{x}_2)=
\left( \begin{array}{c} V_0+V_1 \widehat{x}_1+V_2 \widehat{x}_1\widehat{x}_2+V_3 \widehat{x}_2\\ W_0+W_1 \widehat{x}_1+W_2 \widehat{x}_1\widehat{x}_2+W_3 \widehat{x}_2 \end{array}\right),$$
where
\begin{displaymath}
\left (\begin{array}{cccc}V_0 & W_0  \\ V_1 &
W_1 \\ V_2 & W_2 \\V_3 & W_3  \end{array} \right )=\frac{1}{4}\left (
\begin{array}{cccc} 1& 1 & 1 & 1 \\ -1 & 1 & 1 & -1 \\ 1 & -1 & 1 & -1 \\ -1 & -1 & 1 & 1
\end{array}\right ) \left ( \begin{array}{cccc} \upsilon_1 & \omega_1 \\ \upsilon_2 & \omega_2 \\ \upsilon_3 & \omega_3 \\\upsilon_4 & \omega_4 \end{array}\right ) .
\end{displaymath}

To describe  the stress approximation of PS element, we abbreviate the symmetric tensor $\bm{\tau}
=\left(\begin{array}{cc}\tau_{11} & \tau_{12}\\ \tau_{12} &
\tau_{22} \end{array}\right)$ to $\bm
\tau=(\tau_{11},\tau_{22},\tau_{12})^T$.  The 5-parameter stress mode of PS element takes  the following form on $\widehat{T}$:
% For convenience , we abbreviate it to $\bm{\tau}=(\tau_{11},\tau_{22},\tau_{33})^T$ the 5-parameters stress mode on $\widehat{T}$ for the PS finite element takes the form
\begin{equation}
\label{tau basis}
\widehat{\bm{\tau}}=\left(\begin{array}{c} \widehat{\tau}_{11} \\ \widehat{\tau}_{22} \\ \widehat{\tau}_{12}\end{array}\right)=\left(\begin{array}{ccccc}1 & 0 & 0 &  \widehat{x}_2 & \frac{a_3^2}{b_3^2} \widehat{x}_1 \\ 0 & 1 & 0 & \frac{b_1^2}{a_1^2} \widehat{x}_2&  \widehat{x}_1 \\ 0 & 0 & 1 & \frac{b_1}{a_1} \widehat{x}_2 & \frac{a_3}{b_3} \widehat{x}_1 \end{array}\right) {\bm{\beta}}^{\tau} ~~~~ \text{for}~ ~\bm{\beta}^{\tau}=(\beta^{\tau}_1 ,...,\beta^{\tau}_5)^T\in R^5.
\end{equation}

Then the corresponding    stress approximation space for the PS finite element is
\begin{equation}
\label{P_h}
{\Sigma_D}_h:=\{\bm{\tau} \in {\Sigma_D} : \widehat{\bm{\tau}}=\bm{\tau}|_T o F_T ~~\text{is}~~\text{of} ~~\text{form}~~(\ref{tau basis}),~\forall T \in \mathcal{T}_h\}.
\end{equation}

\subsection{Stochastic hybrid stress  finite element method: $k\times h$-version}\label{k-h}
This subsection  is devoted to the stability and a priori error analysis for the  $k\times h$-version   stochastic hybrid stress  finite element method ($k\times h$-SHSFEM).

\subsubsection{ $k\times h$-SHSFEM scheme}
We first use the same notations as in \cite{babuska 2004} to introduce  a $k$-version  tensor product  finite element space on the stochastic field $\Gamma=\prod_{n=1}^N\Gamma_n\subset R^N$.

Consider a partition of $\Gamma$ %=\prod_{n=1}^N\Gamma_n\subset R^N$
consisting of a finite number of disjoint $R^N$-boxes, $\gamma=\prod_{n=1}^N (a_n^{\gamma}, b_n^{\gamma})$ with $(a_n^{\gamma} , b_n^{\gamma})\subset \Gamma_n$ and   the mesh parameter   $k_n:=\max_{\gamma}|b_n^{\gamma}-a_n^{\gamma}|$ for $n=1,2,\cdots,N$.

%For $n=1,2,\cdots,N$, let $(a_n^{(1)} , b_n^{(1)}), (a_n^{(2)} , b_n^{(2)}),\cdots, (a_n^{(m_n)} , b_n^{(m_n)})$ be $m_n$ disjoint intervals with  $\Gamma_n=\bigcup_{i_n=1}^{i_n=m_n}(a_n^{i_n} , b_n^{i_n})$, and denote
% $k_n:=\max_{\gamma}|b_n^{\gamma}-a_n^{gamma}|$, $\textbf{k}:=(k_1,k_2,...,k_N)$.
%% $\Gamma=\prod_{n=1}^N\Gamma_n$  we make a partition of  each direction,  i.e. ,  and Denote
%% $k_n:=\max_{1\leq i_n\leq m_n}|b_n^{i_n}-a_n^{i_n}|$ for $n=1,2,...,N$  and $\textbf{k}:=(k_1,k_2,...,k_N)$.
% Then the partition of $\Gamma$ are  the tensor product of each element of $\Gamma_n, n=1,2,...,N$.
% , which we denote it $\Gamma_\textbf{k}$, are  the tensor product of each element of $\Gamma_n, n=1,2,...,N$, namely
% $$\Gamma_{\textbf{k}}=\{\prod_{n=1}^N( a_n^{i_n} , b_n^{i_n} ),  i_n=1,...,m_n\}.$$
 Let $\textbf{q}=(q_1,q_2,...,q_N)$ be a nonnegative integer muti-index. We define the $k-$version tensor product finite element space $Y_{\textbf{k}}^{\textbf{q}}$ as %and define the following piecewise polynomial  finite element spaces:
 \begin{equation}
 \label{definition  of Y_k^q}
 Y_{\textbf{k}}^{\textbf{q}}:=\otimes_{n=1}^N Y_{k_n}^{q_n},\quad Y_{k_n}^{q_n}:=\left\{\varphi: \Gamma_n\rightarrow R:
 \varphi |_{(a_n^{\gamma} , b_n^{\gamma})} \in span\{y_n^\alpha: \alpha=0,1,...,q_n\}, \forall \gamma\right\}.
\end{equation}
% \begin{equation}
% \label{definition  of Y_k^q}
% Y_{\textbf{k}}^{\textbf{q}}:=\{\varphi \in L^2(\Gamma):  \varphi|_{\Gamma_e}\in span(\prod_{n=1}^N y_n^{\alpha_n} : \alpha_n \in \mathbb{N}~~ and~~ \alpha_n \leq q_n), \Gamma_e \in  \Gamma_{\textbf{k}}  \}.
%\end{equation}
%We note that $Y_{\textbf{k}}^{\textbf{q}}$ is  called a $k-$version tensor product finite element space \cite{babuska 2004}.
%For given $\textbf{q}$, one can refine the mesh so as to obtain  higher accuracy.

The     $k\times h$-SHSFEM  scheme   for the original weak problem (\ref{continous formulation}), or the modified weak problem  (\ref{truncated continous formulation N}), reads as:
find $(\bm{\sigma}_{kh}, \textbf{u}_{kh})\in  (Y_{\textbf{k}}^{\textbf{q}} \otimes {\Sigma_D}_h )\times  (Y_{\textbf{k}}^{\textbf{q}} \otimes {V_D}_h)  $  such that
\begin{equation}
\label{discrete formulation}
\left\{
    \begin{array}{ll}
        a_N(\bm{\sigma}_{kh},\bm{\tau}_{kh})-b_N(\bm{\tau}_{kh}, \textbf{u}_{kh})=0,
        &    \forall \bm{\tau}_{kh}\in  Y_{\textbf{k}}^{\textbf{q}} \otimes {\Sigma_D}_h, \\
        b_N(\bm{\sigma}_{kh}, \textbf{v}_{kh})=\ell_N(\textbf{v}_{kh}),    & \forall \textbf{v}_{kh}\in Y_{\textbf{k}}^{\textbf{q}} \otimes {V_D}_h.
      \end{array}
       \right.
\end{equation}
Here we recall that
$$ Y_{\textbf{k}}^{\textbf{q}} \otimes {\Sigma_D}_h =span \{ \varphi(\textbf{y})\bm\tau(\textbf{x}) : \varphi\in Y_{\textbf{k}}^{\textbf{q}} ,\bm\tau \in{\Sigma_D}_h \},$$
$$ Y_{\textbf{k}}^{\textbf{q}} \otimes {V_D}_h =span \{\varphi(\textbf{y})\textbf{v}(\textbf{x}) : \varphi\in Y_{\textbf{k}}^{\textbf{q}} , \textbf{v} \in{V_D}_h \},$$
and     ${V_D}_h$, ${\Sigma_D}_h$ are defined   in   \eqref{Q_h}, \eqref{P_h}, respectively.

\subsubsection{Stability   }

To show the  $k\times h$-SHSFEM scheme \eqref{discrete formulation} admits a unique solution, we need some stability conditions.
We note that the continuity of $a_N(\cdot,\cdot)$, $b_N(\cdot,\cdot)$ and  $\ell_N(\cdot)$ follows from their definitions.  Then, according to the theory of mixed methods  \cite{book hybrid}, it suffices to prove the following two discrete versions of the stability conditions.

$(\textbf{A}_h)$ Discrete Kernel-coercivity : for any $\bm{\tau}_{kh} \in Z_{kh}^0 :=\{\bm{\tau}_{kh} \in Y_{\textbf{k}}^{\textbf{q}} \otimes {\Sigma_D}_h : b_N(\bm{\tau}_{kh},\textbf{v}_{kh})=0, \ \forall\textbf{v}_{kh}\in Y_{\textbf{k}}^{\textbf{q}} \otimes {V_D}_h \}$
 , it holds:
\begin{equation}
\label{A_h}
 ||\bm{\tau}_{kh}||^2_{\widetilde{0}}\lesssim a_N(\bm{\tau}_{kh} , \bm{\tau}_{kh}).
\end{equation}

    $(\textbf{B}_h)$ Discrete inf-sup condition : \text{for} \text{any} $\textbf{v}_{kh}\in Y_{\textbf{k}}^{\textbf{q}} \otimes {V_D}_h$ , it holds
\begin{equation}
\label{B_h}
  |\textbf{v}_{kh}|_{\widetilde{1}} \lesssim  \sup_{0\neq \bm{\tau}_{kh} \in Y_{\textbf{k}}^{\textbf{q}} \otimes {\Sigma_D}_h} \frac{b_N(\bm{\tau}_{kh},\textbf{v}_{kh})}{||\bm{\tau}_{kh}||_{\widetilde{0}}}.
\end{equation}

% Introduce the spaces
%$$W_h:=\{w\in L^2(D) : \widehat{w}=w|_T \circ F_T \in span\{1,\hat{x}_1,\hat{x}_2\} , \ \forall T \in \mathcal{T}_h\},$$
%$$\overline{W}_h:=\{w\in W_h: \overline{w}|_T   \text{ is a constant}, \ \forall T\in \mathcal{T}_h\}.$$
To prove the stability condition $(\textbf{A}_h)$, we need the following lemma   \cite{yu guo zhu}:
\begin{lem}\label{lemma assumption}
 %Let the partition $\mathcal{T}_h$ satisfy the shape-regularity condition (\ref{domain regular})\textbf{}.
 Assume that for any piecewise constant function $w$, i.e. $w\in L^2(D)$ with $w|_T=const$,  $\forall T \in \mathcal{T}_h$, there exists   $\textbf{v}\in V_{Dh}$ with
$$||w||^2_0\lesssim \int_D w\text{div} \textbf{v}~\mathrm{d}\textbf{x} ,\quad |\textbf{v}|^2_1\lesssim ||w||^2_0.$$
Then,
for  any  $\bm{\tau}_h \in \{\bm{\tau}_h \in \Sigma_{Dh}: \int_D \bm{\tau}_h : \epsilon(\textbf{v}_h)\mathrm{d}\textbf{x}=0,~~ \forall \textbf{v}_h \in  V_{Dh} \}$, it holds
$$||\bm{\tau}_h||_{0}^2\lesssim  \int_D\frac{1}{\widetilde{E}_{N}} \bm{\tau}_{h} : {{\textbf{C}}}^{-1}\bm{\tau}_{h}  \mathrm{d}\textbf{x}.$$
\end{lem}

We note that the assumption of this lemma, which was first used in \cite{yu 37} in the analysis of  several  quadrilateral nonconforming elements for incompressible elasticity, requires that the quadrilateral mesh   is stable for the Stokes element Q1-P0.
As we know, the only unstable case for Q1-P0 is the checkerboard mode. Thereupon, any quadrilateral mesh subdivision of $D$ which breaks the checkerboard mode is sufficient for the uniform stability $(\textbf{A}_h)$.
%\begin{lem}
%\label{lemma assumption}
% Let the partition $\mathcal{T}_h$ satisfy the shape-regularity condition (\ref{domain regular})\textbf{}.Assume that for any $w\in \overline{W}_h$, there exists some $\textbf{v}\in {V_D}_h$ with
%$||\overline{w}||^2_0\lesssim \int_D \overline{w}~ div \textbf{v}~\mathrm{d}\textbf{x} $ and $|\textbf{v}|^2_1\lesssim ||\overline{w}||^2_0$.
%Then it holds
%$$||w||_0\lesssim \sup_{\textbf{v}\in {V_D}_h,\textbf{v}^b\in B_h} \frac{\int_D w(div \textbf{v}+\widetilde{div} \textbf{v}^b)\mathrm{d}\textbf{x}}{|\textbf{v}+\textbf{v}^b|_{1,h}}~~~~for~ all ~w\in W_h,$$
%where $B_h:=\{\textbf{v}^b\in (L^2(D))^2: \widehat{\textbf{v}}^b(\widehat{x}_1 , \widehat{x}_2)=\textbf{v}^b|_T\circ F_T \in span\{\widehat{x}_1^2-1, \widehat{x}_2^2-1\}^2~~for ~all~T\in \mathcal{T}_h\}$ , the semi-norm $|\cdot|_{1,h}$ on $V_H+B_h$ is defined as $|\cdot|_{1,h}:=(\sum_{T\in \mathcal{T}_h}|\cdot|^2_{1,T})^{\frac{1}{2}}$.
%\end{lem}

%\begin{equation}
%\label{a_N definition}
%a_N(\bm{\sigma}_N,\bm{\tau})
%=\int_{\Gamma}\rho(\textbf{y})\int_D\frac{1}{\widetilde{E}_{N}} \cdot \bm{\sigma}_N : {{\textbf{C}}}^{-1}\bm{\tau}\mathrm{d}\textbf{x} \mathrm{d}\textbf{y},
%\end{equation}
%\begin{equation}
%\label{b_N definition}
%b_N(\bm{\tau},\textbf{u}_N)=\int_\Gamma\rho(\textbf{y})\int_D \bm{\tau} : \epsilon(\textbf{u}_N)\mathrm{d}\textbf{x} \mathrm{d}\textbf{y},
%\end{equation}
\begin{lem}\label{4.2}
Under the same condition as in Lemma \ref{lemma assumption},  the uniform discrete kernel-coercivity condition $(\textbf{A}_h)$ holds.
\end{lem}
\begin{proof}
For any $\bm{\tau}_{kh} \in Z_{kh}^0$,  due to the definitions of spaces $Y_{\textbf{k}}^{\textbf{q}} \otimes {\Sigma_D}_h$ and $Y_{\textbf{k}}^{\textbf{q}} \otimes {V_D}_h$ we easily have   $\bm{\tau}_{kh}(\cdot, \textbf{y}') \in \{\bm{\tau}_h \in {\Sigma_D}_h: \int_D \bm{\tau}_h : \epsilon(\textbf{v}_h)\mathrm{d}\textbf{x}=0,~~ \forall \textbf{v}_h \in  {V_D}_h \}$  for any $\textbf{y}'\in \Gamma$. From Lemma \ref{lemma assumption} it follows
\begin{equation}
\int_D\bm{\tau}_{kh}(\cdot, \textbf{y}'):\bm{\tau}_{kh}(\cdot, \textbf{y}')\mathrm{d}\textbf{x} \lesssim \int_D\frac{1}{\widetilde{E}_{N}(\cdot, \textbf{y}')} ~ \bm{\tau}_{kh}(\cdot, \textbf{y}') : {{\textbf{C}}}^{-1}\bm{\tau}_{kh}(\cdot, \textbf{y}')\mathrm{d}\textbf{x}, ~~\forall \textbf{y}'\in \Gamma,
\end{equation}
%which, by  multiplying $\rho(\textbf{y}')$ on both sides and integral in domain $\Gamma$,
 %,
%Since  $$a_h(\bm{\tau}_h,\bm{\tau}_h)
%=\int_{\Omega}\int_D\frac{1}{\widetilde{E}}_{S_1} \cdot \bm{\sigma}_h : {{\textbf{C}}}^{-1}\bm{\tau}_h\mathrm{d}\textbf{x} \mathrm{d}P(\theta)=\int_{\Omega}\int_D \frac{1}{2\mu'} \bm{\tau}^D:\bm{\tau}^D+\frac{1}{4(\mu'+\lambda')}tr\bm{\tau}tr\bm{\tau} \mathrm{d}\textbf{x} \mathrm{d}P(\theta),$$
%where $\mu'=\frac{\widetilde{E}_{s_1}}{2(1+\nu)}$,$\lambda'=\frac{\widetilde{E}_{s_1}v}{(1+v)(1-2v)}$ for plane strain problems and $\mu'=\frac{\widetilde{E}_{s_1}}{2(1+v)}$,$\lambda'=\frac{\widetilde{E}_{s_1}}{(1+v)(1-v)}$ for plane stress problems.
%
%  Follow the same lines in the prove of Lemmas 2.4 and 2.5 in \cite{33 yu},  we only need to prove that  $||tr\bm{\tau}_h||_{\widetilde{0}}\lesssim ||\bm{\tau}_h^D||_{\widetilde{0}}$ ~~for  any $\bm{\tau}\in \Sigma_h$. According to \cite{yu guo zhu}, we know that for a.e. $\theta\in\Omega$, $$||tr\bm{\tau}(\cdot,\theta)||_0\lesssim ||\bm{\tau}(\cdot,\theta)^D||_0,$$
which immediately  implies $(\textbf{A}_h)$.
\end{proof}
To prove  the discrete inf-sup condition $\textbf{B}_h$ we need the following lemma:
\begin{lem}
\label{lemma discrete condition B_h}
For any $\textbf{v}_{kh}\in Y_{\textbf{k}}^{\textbf{q}} \otimes {V_D}_h$, there exists  $\bm{\tau}_{kh} \in Y_{\textbf{k}}^{\textbf{q}} \otimes {\Sigma_D}_h$ such that, for any $T\in \mathcal{T}_h$,
\begin{equation}
\label{lemma}
\int_{\Gamma}\rho(\textbf{y})\int_T \bm{\tau}_{kh} : \epsilon( \textbf{v}_{kh})\mathrm{d}\textbf{x}\mathrm{d}\textbf{y}=||\bm{\tau}_{kh}||^2_{\widetilde{0},T}\gtrsim || \epsilon(\textbf{v}_{kh})||^2_{\widetilde{0},T}.
\end{equation}
\end{lem}

\begin{proof}
The desired result is immediate from Lemma 4.4 in \cite{yu guo zhu}.
\end{proof}

\begin{lem}\label{4.4}
%Let the partition $\mathcal{T}_h$ satisfy the shape-regularity condition (\ref{domain regular}). Then t
The uniform discrete inf-sup condition $(\textbf{B}_h)$ holds.
\end{lem}
\begin{proof}
From Lemma \ref{lemma discrete condition B_h}, for any $\textbf{v}_{kh}\in Y_{\textbf{k}}^{\textbf{q}} \otimes {V_D}_h$, there exists $\bm{\tau}_{kh} \in Y_{\textbf{k}}^{\textbf{q}} \otimes {\Sigma_D}_h$ such that
$$||\bm{\tau}_{kh}||_{\widetilde{0}} |\textbf{v}_{kh}|_{\widetilde{1}} \lesssim (\sum_T\int_{\Gamma}\rho(\textbf{y})\int_T \bm{\tau}_{kh} : \bm{\tau}_{kh} \mathrm{d}\textbf{x} \mathrm{d}\textbf{y} )^{\frac{1}{2}} (\sum_T\int_{\Gamma}\rho(\textbf{y})\int_T \epsilon(\textbf{v}_{kh}) : \epsilon(\textbf{v}_{kh}) \mathrm{d}\textbf{x}\mathrm{d}\textbf{y})^{\frac{1}{2}} $$
$$\lesssim \sum_T\int_{\Gamma}\rho(\textbf{y})\int_T \bm{\tau}_{kh} : \bm{\tau}_{kh} \mathrm{d}\textbf{x} \mathrm{d}\textbf{y} \lesssim \int_{\Gamma}\rho(\textbf{y})\int_D \bm{\tau}_{kh} : \epsilon( \textbf{v}_{kh})\mathrm{d}\textbf{x}\mathrm{d}\textbf{y},$$
where in the first inequality the equivalence of the seminorm $|\epsilon(\cdot)|_{\widetilde{0}}$ and the norm $||\cdot||_{\widetilde{1}}$ on the space $L^2_P(\Omega ;  ~{V_D})$ is used. Then the uniform discrete inf-sup condition $(\textbf{B}_h)$ follows from
$$|\textbf{v}_{kh}|_{\widetilde{1}}\lesssim \frac{\int_{\Gamma}\rho(\textbf{y})\int_T \bm{\tau}_{kh} : \epsilon(\textbf{v}_{kh})\mathrm{d}\textbf{x}\mathrm{d}\textbf{y}} {||\bm{\tau}_{kh}||_{\widetilde{0}}}\leqslant
\sup_{\bm{\tau}_{kh}'\in Y_{\textbf{k}}^{\textbf{q}} \otimes {\Sigma_D}_h}\frac{\int_{\Gamma}\rho(\textbf{y})\int_T \bm{\tau}_{kh}': \epsilon(\textbf{v}_{kh})\mathrm{d}\textbf{x}\mathrm{d}\textbf{y}}{||\bm{\tau}_{kh}'||_{\widetilde{0}}}.$$
\end{proof}

In light of Lemma \ref{4.2} and Lemma \ref{4.4}, we immediately obtain the following  existence and uniqueness of the $k\times h$-SHSFEM approximation $(\bm{\sigma}_{kh} , \textbf{u}_{kh})$:
\begin{thm}\label{th4.1}
Under the same condition as in Lemma \ref{lemma assumption} , the  discretization problem (\ref{discrete formulation}) admits a unique solution $(\bm{\sigma}_{kh} , \textbf{u}_{kh}) \in (Y_{\textbf{k}}^{\textbf{q}} \otimes {\Sigma_D}_h)\times (Y_{\textbf{k}}^{\textbf{q}} \otimes {V_D}_h)$.
\end{thm}

\subsubsection{Uniform error estimation}\label{sec4.2.3}

In what follows  we shall derive a priori  estimates of the errors $ ||\bm{\sigma}-\bm{\sigma}_{kh}||_{\widetilde{0}}$ and $|\textbf{u}-\textbf{u}_{kh}|_{\widetilde{1}}$ which are uniform with respect to the Lam$\acute{e}$ constant $\lambda\in (0, +\infty)$, where $(\bm{\sigma}, \textbf{u})\in (L^2_P(\Omega;~ \Sigma_D))\times( L^2_P(\Omega ;  ~{V_D}))$ is the solution of the weak problem (\ref{continous formulation}).

Let $(\bm{\sigma}_N, \textbf{u}_N)\in (L^2_{\rho}(\Gamma)\otimes {\Sigma_D})\times( L^2_{\rho}(\Gamma)\otimes {V_D})$ be the solution of  truncated weak problem (\ref{truncated continous formulation N}).   By  triangle inequality it holds
 \begin{equation}
 \label{412}
 ||\bm{\sigma}-\bm{\sigma}_{kh}||_{\widetilde{0}} \leq ||\bm{\sigma}-\bm{\sigma}_N||_{\widetilde{0}}
 +||\bm{\sigma}_N-\bm{\sigma}_{kh}||_{\widetilde{0}},
 \end{equation}
 \begin{equation}
 |\textbf{u}-\textbf{u}_{kh}|_{\widetilde{1}} \leq |\textbf{u}-\textbf{u}_N|_{\widetilde{1}}
 +|\textbf{u}_N-\textbf{u}_{kh}|_{\widetilde{1}},
 \end{equation}
where the perturbation errors,  $||\bm{\sigma}-\bm{\sigma}_N||_{\widetilde{0}}$ and $ |\textbf{u}-\textbf{u}_N|_{\widetilde{1}}$, are estimated by   Lemma \ref{propo truncated N}.  For the finite element approximation error terms $||\bm{\sigma}_N-\bm{\sigma}_{kh}||_{\widetilde{0}}$ and $|\textbf{u}_N-\textbf{u}_{kh}|_{\widetilde{1}}$,   from the stability $(\textbf{A}_h)$, $(\textbf{B}_h)$ and  the  standard theory of mixed finite element methods \cite{book hybrid} it follows
\begin{equation}\label{415}
||\bm{\sigma}_N-\bm{\sigma}_{kh}||_{\widetilde{0}}+|\textbf{u}_N-\textbf{u}_{kh}|_{\widetilde{1}}
\lesssim  \inf_{\bm{\tau}_{kh}\in Y_{\textbf{k}}^{\textbf{q}} \otimes {\Sigma_D}_h} ||\bm{\sigma}_N-\bm{\tau}_{kh}||_{\widetilde{0}}+\inf_{\textbf{v}_{kh}\in Y_{\textbf{k}}^{\textbf{q}} \otimes {V_D}_h}|\textbf{u}_N-\textbf{v}_{kh}|_{\widetilde{1}}.
\end{equation}

To further estimate the righthand-side terms of the above inequality, we need some regularity of the solution $(\bm{\sigma}_N, \textbf{u}_N)$.
In fact, it is well-known that the following regularity holds:
 \begin{equation}\label{regu1}
 ||\bm{\sigma}_N(\cdot, \textbf{y})||_{1}+||\textbf{u}_N(\cdot, \textbf{y})||_{2}\lesssim
 ||\textbf{f}_N(\cdot, \textbf{y})||_{0}+||\textbf{g}_N(\cdot, \textbf{y})||_{0,\partial D_1},\quad \forall \textbf{y} \in \Gamma.
 \end{equation}
On the other hand, in view of  \eqref{truncated E bounded} and the truncated K-L expansions \eqref{truncated K-L}-\eqref{truncated K-L of g}, and by taking
derivatives with respect to $y_n$ in (\ref{deterministic problem}),  standard
  inductive arguments yield
\begin{equation}
\label{extension derivative q_n}
\frac{||\partial_{y_n}^{q_n+1}\bm{\sigma}_N(\cdot, \textbf{y})||_{0}}{(q_n+1)!}+\frac{|\partial_{y_n}^{q_n+1}\textbf{u}_N(\cdot, \textbf{y})|_{1}}{(q_n+1)!}\lesssim (2\gamma_n)^{q_n+1}(||\textbf{f}_N(\cdot,\textbf{y})||_{0}+||\textbf{g}_N(\cdot,\textbf{y})||_{0,\partial D_1}+1),\quad \forall \textbf{y} \in \Gamma,
\end{equation}
where \begin{equation}\label{gamma_n}
\gamma_n:=\max \{
\frac{1}{{e}_{min}'}\sqrt{\widetilde{\lambda}_n}||\widetilde{b}_n||_{L^{\infty}(D)}, \sqrt{\widehat{\lambda}_{in}}||\widehat{b}_{in}||_{0}(i=1,2), \sqrt{\overline{\lambda}_{in}}||\overline{b}_{in}||_{0,\partial D_1}(i=1,2) \}.
\end{equation}%$0<\tau_n<\frac{1}{2\gamma_n}$.
 Then, thanks to $Y_{\textbf{k}}^{\textbf{q}}=\otimes_{n=1}^N Y_{k_n}^{q_n}$ and the regularity \eqref{regu1}-\eqref{extension derivative q_n},  standard   interpolation estimation  yields
%
%assume  $\bm{\sigma}_N=(\sigma_{N,ij})$ and $\textbf{u}_N=(u_{N,i})$ ($i,j=1,2$) to satisfy  the regularity conditions
%\begin{equation} \label{regularity}
%{\sigma}_{N,ij} \in  H^{\textbf{q}+1}(\Gamma)\otimes H^1(D), \quad u_{N,i}\in H^{\textbf{q}+1}(\Gamma)\otimes H^2(D),
%\end{equation}
% where
%$H^{\textbf{q}+1}(\Gamma):=H^{q_1+1}(\Gamma_1)\otimes H^{q_2+1}(\Gamma_2)\otimes...\otimes H^{q_N+1}(\Gamma_N) $
%for $\textbf{q}=(q_1,q_2,\cdots, q_N)$.
%Due to $Y_{\textbf{k}}^{\textbf{q}}=\otimes_{n=1}^N Y_{k_n}^{q_n}$ and \eqref{regularity}, standard estimation shows
\begin{eqnarray}
\label{sigma-tau_h}
 \inf_{\bm{\tau}_{kh}\in Y_{\textbf{k}}^{\textbf{q}} \otimes {\Sigma_D}_h} ||\bm{\sigma}_N-\bm{\tau}_{kh}||_{\widetilde{0}}&\lesssim&  h||\bm{\sigma}_N||_{\widetilde{1}}+\sum_{n=1}^{N} (\frac{k_n}{2})^{q_n+1}\frac{||\partial_{y_n}^{q_{n+1}}\bm{\sigma}_N||_{L^2(\Gamma)\otimes {\Sigma_D}}}{(q_n+1)!}\nonumber\\
 &\lesssim&  h+\sum_{n=1}^{N} ( {k_n\gamma_n} )^{q_n+1},\\
 \label{u-u_h}
\inf_{\textbf{v}_{kh}\in Y_{\textbf{k}}^{\textbf{q}} \otimes {V_D}_h}|\textbf{u}_N-\textbf{v}_{kh}|_{\widetilde{1}}&\lesssim &h||\textbf{u}_N||_{\widetilde{2}}+\sum_{n=1}^{N} (\frac{k_n}{2})^{q_n+1}
\frac{||\partial_{y_n}^{q_{n+1}}\textbf{u}_N||_{L^2(\Gamma)\otimes {V_D}}}{(q_n+1)!}\nonumber\\
&\lesssim &h+\sum_{n=1}^{N} ( {k_n\gamma_n} )^{q_n+1}.
 \end{eqnarray}

 In light of the estimates \eqref{415} and \eqref{sigma-tau_h}-\eqref{u-u_h},  we immediately obtain the following conclusion.

\begin{thm}\label{theorem42}
Let   $(\bm{\sigma}_N, \textbf{u}_N)\in (L^2_{\rho}(\Gamma)\otimes {\Sigma_D})\times( L^2_{\rho}(\Gamma)\otimes {V_D})$ and $(\bm{\sigma}_{kh}, \textbf{u}_{kh})\in (Y_{\textbf{k}}^{\textbf{q}} \otimes {\Sigma_D}_h)\times( Y_{\textbf{k}}^{\textbf{q}} \otimes {V_D}_h)$ be the solutions of  (\ref{truncated continous formulation N}) and (\ref{discrete formulation}), respectively. Then, under the same condition as in Lemma \ref{lemma assumption}  and for sufficiently large $N$, it holds
\begin{equation}\label{error1}
||\bm{\sigma}_N-\bm{\sigma}_{kh}||_{\widetilde{0}}+|\textbf{u}_N-\textbf{u}_{kh}|_{\widetilde{1}}
\lesssim h+\sum_{n=1}^{N} ( {k_n\gamma_n} )^{q_n+1}.
\end{equation}
\end{thm}
\begin{rem}\label{Rem41}
We notice that the estimate \eqref{error1} is optimal with respect to the mesh parameters $h$ and $\textbf{k}=(k_1,k_2,\cdots,k_N)$, but not optimal with respect to the polynomial  degree $\textbf{q}=(q_1,q_2,\cdots,q_N)$ since it requires  $k_n\gamma_n<1$.
\end{rem}

The above theorem, together with Lemma \ref{propo truncated N}, implies the following a priori error estimates for the $k\times h$-SHSFEM approximation $(\bm{\sigma}_{kh}, \textbf{u}_{kh})$.
\begin{thm}
Let $(\bm{\sigma}, \textbf{u})\in (L^2_P(\Omega;~ \Sigma_D))\times( L^2_P(\Omega ;  ~{V_D}))$ and $(\bm{\sigma}_{kh}, \textbf{u}_{kh})\in (Y_{\textbf{k}}^{\textbf{q}} \otimes {\Sigma_D}_h)\times( Y_{\textbf{k}}^{\textbf{q}} \otimes {V_D}_h)$ be the solutions of (\ref{continous formulation}) and (\ref{discrete formulation}), respectively.  Then, under the same conditions as in Theorem \ref{theorem42},  it holds
\begin{equation}
||\bm{\sigma}-\bm{\sigma}_{kh}||_{\widetilde{0}}+|\textbf{u}-\textbf{u}_{kh}|_{\widetilde{1}}
\lesssim   N^{1/2}e^{- r N^{1/2}} + h +
\sum_{n=1}^{N}(k_n\gamma_n)^{q_n+1}
\end{equation}
  for any $r>0$ if the  covariance functions of $\widetilde{E}$, $\textbf{f}$ and $\textbf{g}$
are piecewise analytic,  and holds
\begin{equation}
||\bm{\sigma}-\bm{\sigma}_{kh}||_{\widetilde{0}}+|\textbf{u}-\textbf{u}_{kh}|_{\widetilde{1}}
\lesssim   N^{-s}+ h +
\sum_{n=1}^{N}(k_n\gamma_n)^{q_n+1}
\end{equation}
 for any $s>0$ if the  covariance functions of $\widetilde{E}$, $\textbf{f}$ and $\textbf{g}$
are piecewise smooth.
%where $\gamma_n=\max \{
%\sqrt{\widetilde{\lambda}_n}||\widetilde{b}_n||_{L^{\infty}(D)}/{\underline{e}_{N}}, \sqrt{\widehat{\lambda}_n}||\widehat{b}_n||_{0}, \sqrt{\overline{\lambda}_n}||\overline{b}_n||_{0,\Gamma_n} \}$.
\end{thm}
\begin{rem}
 Here we recall that $"\lesssim"$ denotes $"\leq C " $ with C a positive constant independent of  $\lambda$ , $h$ , $N$,  $\textbf{k}$.
\end{rem}

 \subsection{Stochastic hybrid stress finite element approximation:  $p\times h$ version}

As shown in Section \ref{k-h} and Remark \ref{Rem41}, the $k\times h$-SHSFEM is based on the $k$ partition of the stochastic field  $\Gamma$ and requires the mesh parameter   $k_n$ ($n=1,2,\cdots,N$) to be sufficiently small so as to acquire optimal error estimates.

In this subsection, we shall introduce a $p\times h$-version   stochastic hybrid stress  finite element method ($p\times h$-SHSFEM), which     does not  require to refine $\Gamma$. We will show  this method  is of   exponential rates of convergence with
respect to   the degrees of the polynomials used for approximation. To this end,
we first assume
\begin{equation}\label{assump}
\widetilde{E}_N\in C^0(\Gamma, L^{\infty}(D)),\quad \textbf{f}_N\in C^0(\Gamma, L^2(D)),\quad \textbf{g}_N\in C^0(\Gamma, L^2(\partial D_1)).
\end{equation}
Here
\begin{equation}
C^0(\Gamma, B):=\{ v:\Gamma \rightarrow B,  v~  \text{is continuous in}~
 \textbf{y} ~\text{and} ~ \max_{\textbf{y}\in \Gamma}||v(\textbf{y}) ||_B <+\infty\}
\end{equation}
for   any  Banach space, $B$, of functions defined in $D$.
The  above assumptions indicate that the solution, $(\bm{\sigma}_N,\textbf{u}_N )$,  of  the  problem \eqref{truncated continous formulation N},  satisfies
$$\bm{\sigma}_N \in  C^0(\Gamma, {\Sigma_D}),\quad  \textbf{u}_N \in C^0(\Gamma, {V_D}).$$

%This subsection  is devoted to the analysis of the  .

Let $\textbf{p}:=(p_1,p_2,...,p_N)$ be a  nonnegative integer muti-index. We define the   $p-$version tensor product finite element  space $Z^{\textbf{p}}$  as
\begin{equation}
Z^{\textbf{p}}:=\otimes_{n=1}^N Z_n^{p_n},\quad  Z_n^{p_n}:=  \left\{\varphi: \Gamma_n\rightarrow R: \varphi \in span\{y_n^\alpha: \alpha=0,1,...,p_n\} \right\}.
\end{equation}
% where the nonnegative integer muti-index $\textbf{p}:=(p_1,p_2,...,p_N)$  plays the role of $\textbf{q}$ in (\ref{definition  of Y_k^q}).
 Then the  $p\times h$-SHSFEM scheme   reads as:
find $(\bm{\sigma}_{ph}, \textbf{u}_{ph})\in  (Z^{\textbf{p}} \otimes {\Sigma_D}_h) \times  (Z^{\textbf{p}} \otimes {V_D}_h)  $  such that
\begin{equation}
\label{discrete formulation  ph-version}
\left\{
    \begin{array}{ll}
        a_N(\bm{\sigma}_{ph},\bm{\tau}_{ph})-b_N(\bm{\tau}_{ph}, \textbf{u}_{ph})=0,
        ~~~~~~        \forall \bm{\tau}_{ph}\in  Z^{\textbf{p}} \otimes {\Sigma_D}_h, \\
        b_N(\bm{\sigma}_{ph}, \textbf{v}_{ph})=\ell_N(\textbf{v}_{ph}),    ~~~~~~~~~~~~~~~~  \forall \textbf{v}_{ph}\in Z^{\textbf{p}} \otimes {V_D}_h.
      \end{array}
       \right.
\end{equation}

%  one can
%acquire higher accuracy  by increasing the degrees of polynomials, $\textbf{q}$.

%\subsubsection{ $p\times h$-SHSFEM scheme}
 %In this section,  $p\times h$-version SFEM  is used to discrete  (\ref{truncated continous formulation N}), and  uniform error estimate is  derived.
%Define
%$$Z^{\textbf{p}} \otimes {V_D}_h=Z^{\textbf{p}} \otimes {\Sigma_D}_h\equiv\{\phi=\phi( \textbf{x}, \textbf{y}) ~|~ \phi \in span(\varphi(\textbf{y})w(\textbf{x}) : \varphi\in Z^{\textbf{p}} , \textbf{w} \in {\Sigma_D}_h)\},$$
%$$Z^{\textbf{p}} \otimes {V_D}_h=Z^{\textbf{p}} \otimes {V_D}_h\equiv \{\psi=\psi ( \textbf{x},\textbf{y})~ |~ \psi \in span(\chi (\textbf{y})v(\textbf{x}): \chi \in Z^{\textbf{p}},\textbf{ v} \in {V_D}_h)\}.$$

We note that $Z^{\textbf{p}}$ is a special case of the $k-$version tensor product finite element  space $Y_{\textbf{k}}^{\textbf{q}}$, then, in this sense, the $p\times h$-SHSFEM can be viewed as a special case of the $k\times h$-SHSFEM.  As a result,   the corresponding stability conditions and  the existence and uniqueness of the solution  of the  $p\times h$-SHSFEM scheme  (\ref{discrete formulation  ph-version}) follow from those of the $k\times h$-SHSFEM (cf. Lemma \ref{4.2}, Lemma \ref{4.4} and Theorem \ref{th4.1}).

%\subsubsection{Uniform error estimate}
Following the same routine  as in Section \ref{sec4.2.3} (cf. the estimates \eqref{412}-\eqref{415}),    we only need to  estimate the terms $\inf\limits_{\bm{\tau}_{ph}  \in Z^{\textbf{p}} \otimes {\Sigma_D }_h} ||\bm{\sigma}_N-{\bm{\tau}}_{ph}||_{\widetilde{0}}$  and
$\inf\limits_{\textbf{v}_{ph}\in Z^\textbf{p}\otimes {V_D}_h}|\textbf{u}_N-\textbf{v}_{ph}|_{\widetilde{1}}.$
Since
\begin{eqnarray}
\label{9}
\inf\limits_{\bm{\tau}_{ph}  \in Z^{\textbf{p}} \otimes {\Sigma_D }_h} ||\bm{\sigma}_N-{\bm{\tau}}_{ph}||_{\widetilde{0}}\nonumber
&\lesssim&
 \inf\limits_{\bm{\tau}_p  \in Z^{\textbf{p}} \otimes {\Sigma_D }} ||\bm{\sigma}_N-{\bm{\tau}}_p||_{\widetilde{0}}+\inf\limits_{\bm{\tau}_h  \in L^2_\rho(\Gamma) \otimes {\Sigma_D }_h} ||\bm{\sigma}_N-{\bm{\tau}_h}||_{\widetilde{0}}\\
 &\lesssim&   \inf\limits_{\bm{\tau}_{p}  \in Z^{\textbf{p}} \otimes {\Sigma_D }} ||\bm{\sigma}_N-{\bm{\tau}}_p||_{\widetilde{0}}+h||\bm{\sigma}_N||_{\widetilde{1}},
\end{eqnarray}
\begin{eqnarray}
\label{10}
\inf\limits_{\textbf{v}_{ph}\in Z^\textbf{p}\otimes {V_D}_h}|\textbf{u}_N-\textbf{v}_{ph}|_{\widetilde{1}}\nonumber
&\lesssim  &\inf\limits_{\textbf{v}_p\in Z^\textbf{p}\otimes {V_D}}|\textbf{u}_N-\textbf{v}_p|_{\widetilde{1}}+\inf\limits_{\textbf{v}_h\in L^2_\rho(\Gamma)\otimes {V_D}_h}|\textbf{u}_N-\textbf{v}_h|_{\widetilde{1}}\\
 &\lesssim& \inf\limits_{\textbf{v}_p\in Z^\textbf{p}\otimes {V_D}}|\textbf{u}_N-\textbf{v}_p|_{\widetilde{1}}+h||\textbf{u}_N||_{\widetilde{2}},
\end{eqnarray}
%$$ \inf_{\textbf{v}_p\in Z^\textbf{p}\otimes {V_D}}|\textbf{u}_N-\textbf{v}_p|_{\widetilde{1}}.$$
it remains to estimate $\inf\limits_{\bm{\tau}_p  \in Z^{\textbf{p}} \otimes {\Sigma_D }} ||\bm{\sigma}_N-{\bm{\tau}}_p||_{\widetilde{0}}$  and
$\inf\limits_{\textbf{v}_p\in Z^\textbf{p}\otimes {V_D}}|\textbf{u}_N-\textbf{v}_p|_{\widetilde{1}}.$
Recalling $ Z^{\textbf{p}}=\otimes_{n=1}^N Z_n^{p_n}$, we easily have the following estimates:
%   Define  the $L^2$ projection $\Pi_n: L^2(\Gamma_n)\rightarrow Z^{p_n}_{n}$,  for $n=1,2,...,N$, and  denote $\Pi:=\Pi_1\Pi_2...\Pi_N$.
%Then
%\begin{eqnarray}
%\label{k 4}
%\inf\limits_{\bm{\tau}_p  \in Z^{\textbf{p}} \otimes {\Sigma_D }} ||\bm{\sigma}_N-{\bm{\tau}}_p||_{\widetilde{0}}\nonumber
%&\leq&
%\sqrt{||\rho||_{L^{\infty}(\Gamma)}}\inf_{\bm{\tau}_{p}  \in Z^{\textbf{p}}\otimes \Sigma_D} ||\bm{\sigma}_N-{\bm{\tau}}_{p}||_{L^2(\Gamma)\otimes \Sigma_D}\\\nonumber
% &\lesssim& ||\bm{\sigma}_N-\Pi{\bm{\sigma}_N}||_{L^2(\Gamma)\otimes \Sigma_D}\\\nonumber
% &\lesssim& ||\bm{\sigma}_N-\Pi_1{\bm{\sigma}_N}+\Pi_1({\bm{\sigma}_N}-\Pi_2{\bm{\sigma}_N})+
% ...+\\\nonumber
% &~&+\Pi_1\Pi_2...\Pi_{N-1}({\bm{\sigma}_N}-\Pi_N{\bm{\sigma}_N})||_{L^2(\Gamma)\otimes \Sigma_D}\\
% &\lesssim& \sum_{n=1}^N ||\bm{\sigma}_N-\Pi_n{\bm{\sigma}_N}||_{L^2(\Gamma)\otimes \Sigma_D}.
%\end{eqnarray}
%If $\bm{\sigma}_N \in  C^0(\Gamma, {\Sigma_D})$,  for any $\bm{\tau}_{p_n} \in Z_n^{p_n}\otimes \Sigma_D$, it holds that $\Pi_n \bm{\tau}_{p_n}=\bm{\tau}_{p_n}$, and
%\begin{eqnarray}
%\label{ph 1}
%||\bm{\sigma}_N-\Pi_n{\bm{\sigma}_N}||_{L^2(\Gamma)\otimes \Sigma_D}\nonumber
%&\leq&
%||\bm{\sigma}_N-\bm{\tau}_{p_n}||_{L^2(\Gamma)\otimes \Sigma_D}+||\Pi_n(\bm{\tau}_{p_n}-{\bm{\sigma}_N})||_{L^2(\Gamma)\otimes \Sigma_D}\\
%&\lesssim& ||\bm{\sigma}_N-\bm{\tau}_{p_n}||_{C^0(\Gamma, \Sigma_D)}.
%\end{eqnarray}
%Since $\bm{\tau}_{p_n}$  is arbitrary  in  $Z_n^{p_n}\otimes \Sigma_D$,
% combining (\ref{k 4}) (\ref{ph 1}) yield
\begin{equation}
\label{5}
 \inf\limits_{\bm{\tau}_p  \in Z^{\textbf{p}} \otimes {\Sigma_D }} ||\bm{\sigma}_N-{\bm{\tau}}_p||_{\widetilde{0}}
\lesssim\sum_{n=1}^N \inf_{\bm{\tau}_{p_n}\in Z_n^{p_n}\otimes \Sigma_D} ||\bm{\sigma}_N-\bm{\tau}_{p_n}||_{C^0(\Gamma, \Sigma_D)},
\end{equation}
%Similarly, if $ \textbf{u}_N \in C^0(\Gamma, {V_D})$, it holds
\begin{equation}
\label{8}
 \inf\limits_{\textbf{v}_p\in Z^\textbf{p}\otimes {V_D}}|\textbf{u}_N-\textbf{v}_p|_{\widetilde{1}}
\lesssim\sum_{n=1}^N \inf_{\textbf{v}_{p_n}\in Z_n^{p_n}\otimes V_D} ||\textbf{u}_N-\textbf{v}_{p_n}||_{C^0(\Gamma, V_D)}.
\end{equation}
Then the thing left is  to estimate the right hand side terms of the above two inequalities.

Denote $ \Gamma_n^*:=\prod_{i=1,i\neq n}^N \Gamma_i,$  then
$\Gamma= \Gamma_n\times \Gamma_n^*,$ and for any $\textbf{y}\in \Gamma$ we denote $\textbf{y}=(y_n, \textbf{y}_n^*)$ with $y_n\in \Gamma_n$ and $\textbf{y}_n^*\in \Gamma_n^*$.
%and  $y_n^*$  an arbitrary element of $ \Gamma_n^*$.
We have  the following lemma.
\begin{lem}
\label{lemma ph}
Let  $(\bm{\sigma}_N, \textbf{u}_N)\in (L^2_{\rho}(\Gamma)\otimes {\Sigma_D})\times( L^2_{\rho}(\Gamma)\otimes {V_D})$  be the solution of the problem (\ref{truncated continous formulation N}).
Then  for any $\textbf{x}\in D$, $\textbf{y}=(y_n, \textbf{y}_n^*)\in \Gamma$,
 the solutions $\bm{\sigma}_N (x, y_n, y_n^*)$ and $\textbf{u}_N(x,y_n,y_n^*)$ as functions of $y_n$, i.e.  $\bm{\sigma}_N : \Gamma_n \rightarrow C^0( \Gamma_n^*; {\Sigma_D})$,~ $\textbf{u}_N :\Gamma_n \rightarrow C^0( \Gamma_n^*; {V_D})$, can be analytically extended to the complex plane
\begin{equation}\nonumber
\Xi(\Gamma_n; d_n):=\{z \in \mathbb{C}, dist(z, \Gamma_n)\leq d_n  \},
\end{equation}
with $0<d_n<\frac{1}{2\gamma_n}$ and
$\gamma_n$ given by \eqref{gamma_n}.
%$:=\max \{\frac{1}{{e}_{min}'}\sqrt{\widetilde{\lambda}_n}||\widetilde{b}_n||_{L^{\infty}(D)}, \sqrt{\widehat{\lambda}_{in}}||\widehat{b}_{in}||_{0}(i=1,2), \sqrt{\overline{\lambda}_{in}}||\overline{b}_{in}||_{0,\partial D_1} (i=1,2)\}$.
In addition, for all $z \in \Xi(\Gamma_n; d_n)$, it holds
\begin{equation}
\label{p 3}
||\bm{\sigma}_N(z)||_{C^0(\Gamma_n^*;\overline{ \Sigma})}+|\textbf{u}_N(z)|_{C^0(\Gamma_n^*;{V_D})}
\lesssim \frac{1}{1-2d_n\gamma_n}(||\textbf{f}_N||_{C^0(\Gamma;L^2(D))}+||\textbf{g}_N||_{C^0(\Gamma;L^2(\partial D_1))}+1).
\end{equation}
\end{lem}
\begin{proof} Similar to
 (\ref{extension derivative q_n}),  for $\textbf{y}\in \Gamma$, $r\geq 0$ and $ n=1,2,...,N$ it holds
\begin{equation}
\label{extension derivative r}
\frac{||\partial_{y_n}^{r}\bm{\sigma}_N(\cdot, \textbf{y})||_{0}}{r!}+\frac{|\partial_{y_n}^{r}\textbf{u}_N(\cdot, \textbf{y})|_{1}}{r!}\lesssim (2\gamma_n)^r(||\textbf{f}_N(\cdot,\textbf{y})||_{0}+||\textbf{g}_N(\cdot,\textbf{y})||_{0,\partial D_1}+1).
\end{equation}
For any $y_n \in \Gamma_n$, we define power series %$\bm{\sigma}_N : \Gamma_n \rightarrow C^0(\Gamma_n^*, {\Sigma_D})$,~ $\textbf{u}_N :\Gamma_n \rightarrow C^0(\Gamma_n^*, {V_D})$ as
$$
 \bm{\sigma}_N(\textbf{x},z,y_n^*)=\sum_{r=0}^{\infty}\frac{(z-y_n)^r}{r!} \partial_{y_n}^{r}\bm{\sigma}_N(\textbf{x},y_n,y_n^*), ~~~~\textbf{u}_N(\textbf{x},z,y_n^*)=\sum_{r=0}^{\infty}\frac{(z-y_n)^r}{r!} \partial_{y_n}^{r}\textbf{u}_N(\textbf{x},y_n,y_n^*).
$$
then it follows
$$
||\bm{\sigma}_N(\textbf{x},z,y_n^*)||_{0}\leq \sum_{r=0}^{\infty}\frac{|z-y_n|^r}{r!}||\partial_{y_n}^{r}\bm{\sigma}_N(\textbf{x},y_n,y_n^*)||_{0},
$$
$$
|\textbf{u}_N(\textbf{x},z,y_n^*)|_{1}\leq \sum_{r=0}^{\infty}\frac{|z-y_n|^r}{r!}|\partial_{y_n}^{r}\textbf{u}_N(\textbf{x},y_n,y_n^*)|_{1}.
$$
  Due to (\ref{extension derivative r}), we   easily  know that   the  above two series converge for all $z \in \Xi(\Gamma_n; d_n)$. Furthermore, by a continuation argument, the functions
 $\bm{\sigma}_N$, $\textbf{u}_N$ can be extended analytically on the whole region $\Xi(\Gamma_n; d_n)$, and the estimate (\ref{p 3}) follows.
 %Inequality (\ref{p 3}) follows from (\ref{extension derivative r}) with $r=0$.
\end{proof}

%If $\bm{\sigma}_N \in C^0(\Gamma)\otimes\overline{\Sigma }, \textbf{u}_N \in C^0(\Gamma)\otimes{V_D}$, using method
%similar to (\ref{k 4}), we have
%\begin{eqnarray}
%\label{8}
%\inf_{\bm{\tau}_p  \in Z^{\textbf{p}} \otimes \overline{\Sigma }} ||\bm{\sigma}_N-{\bm{\tau}}_p||_{\widetilde{0}} \nonumber
%&\leq&
% \sqrt{||\rho||_{L^{\infty}(\Gamma)}} \sum_{n=1}^N \inf_{\bm{\tau}_{p_n}  \in Z_n^{p_n} \otimes \overline{\Sigma }} ||\bm{\sigma}_N-\bm{\tau}_{p_n}||_{L^2(\Gamma)\otimes {\Sigma_D}}\\
%&\leq&
%\sqrt{||\rho||_{L^{\infty}(\Gamma)}} \sum_{n=1}^N \inf_{\bm{\tau}_{p_n}  \in Z_n^{p_n} \otimes \overline{\Sigma }} ||\bm{\sigma}_N-\bm{\tau}_{p_n}||_{C^0(\Gamma)\otimes {\Sigma_D}}.
%\end{eqnarray}
%\begin{eqnarray}
%\label{5}
%\inf_{\textbf{v}_p\in Z^\textbf{p}\otimes {V_D}}|\textbf{u}_N-\textbf{v}_p|_{\widetilde{1}}\nonumber
%&\leq&
%\sqrt{||\rho||_{L^{\infty}(\Gamma)}}
%  \sum_{n=1}^N \inf_{\textbf{v}_{p_n}  \in Z_n^{p_n} \otimes \overline{V }} ||\textbf{u}_N-\textbf{v}_{p_n}||_{L^2(\Gamma)\otimes {V_D}}\\
%&\leq&
% \sqrt{||\rho||_{L^{\infty}(\Gamma)}}
%  \sum_{n=1}^N \inf_{\textbf{v}_{p_n}  \in Z_n^{p_n} \otimes \overline{V }} ||\textbf{u}_N-\textbf{v}_{p_n}||_{C^0(\Gamma)\otimes {V_D}}.
%\end{eqnarray}
In order to estimate the right-hand-side terms of (\ref{5})(\ref{8}), we need one more lemma by Babu$\breve{s}$ka et al \cite{collocation}.
\begin{lem}
\label{lemma ph2}
Let $B$ be a Banach space, and $L\subset R $   be a bounded set.
 Given a function $v\in C^0(L; B)$ which admits an analytic extension  in the region of the complex plane $\Xi(L; d)=\{z\in \mathbb{C}, dist(z,L)\leq d\}$ for some $d>0$, it holds
 \begin{equation}
 \min_{w\in P_p(L)\otimes B}||v-w||_{C^0(L; B)}\leq \frac{2}{\varrho-1}\varrho^{-p} \max_{z\in \Xi(L; d)}||v(z)||_{B},
 \end{equation}
where  $P_p(L):=span(y^s, s=0,1,...,p)$, $1<\varrho:=\frac{2d}{|L|}+\sqrt{1+\frac{4d^2}{|L|^2}}$.
\end{lem}

 In light of  (\ref{9})-(\ref{8}) and Lemmas \ref{lemma ph}-\ref{lemma ph2}, we immediately  obtain the following result.
 \begin{thm}\label{theorem44}
Let $(\bm{\sigma}_N, \textbf{u}_N)\in (L^2_{\rho}(\Gamma)\otimes {\Sigma_D})\times( L^2_{\rho}(\Gamma)\otimes {V_D})$ and $(\bm{\sigma}_{ph}, \textbf{u}_{ph})\in (Z^{\textbf{p}} \otimes {\Sigma_D}_h)\times( Z^{\textbf{p}} \otimes {V_D}_h)$ be the solutions of  (\ref{truncated continous formulation N}) and (\ref{discrete formulation  ph-version}), respectively. Then, under the same condition as in Lemma \ref{lemma assumption} and for sufficiently large $N$, it holds
\begin{equation}\label{error1}
||\bm{\sigma}_N-\bm{\sigma}_{kh}||_{\widetilde{0}}+|\textbf{u}_N-\textbf{u}_{kh}|_{\widetilde{1}}
\lesssim h+\sum_{n=1}^{N}{\varrho_n}^{-p_n},
\end{equation}
where $\varrho_n=\frac{2 d_n}{|\Gamma_n|}+\sqrt{1+\frac{4 d_n^2}{|\Gamma_n|^2}}$ and $0<d_n<\frac{1}{2\gamma_n}$.
\end{thm}
The above theorem, together with Lemma \ref{propo truncated N}, implies the following a priori error estimates for the $p\times h$-SHSFEM approximation $(\bm{\sigma}_{ph}, \textbf{u}_{ph})$.
\begin{thm}
Let $(\bm{\sigma}, \textbf{u})\in (L^2_P(\Omega;~ \Sigma_D))\times( L^2_P(\Omega ;  ~{V_D}))$  and $(\bm{\sigma}_{ph}, \textbf{u}_{ph})\in (Z^{\textbf{p}} \otimes {\Sigma_D}_h, Z^{\textbf{p}} \otimes {V_D}_h)$ be the solutions of (\ref{continous formulation})  and (\ref{discrete formulation  ph-version}), respectively.  Then, under the same conditions as in Theorem \ref{theorem44},  it holds
\begin{equation}
||\bm{\sigma}-\bm{\sigma}_{kh}||_{\widetilde{0}}+|\textbf{u}-\textbf{u}_{kh}|_{\widetilde{1}}
\lesssim   N^{1/2}e^{- r N^{1/2}} + h +
\sum_{n=1}^{N}{\varrho_n}^{-p_n}
\end{equation}
  for any $r>0$ if the  covariance functions of $\widetilde{E}$, $\textbf{f}$ and $\textbf{g}$
are piecewise analytic,  and holds
\begin{equation}
||\bm{\sigma}-\bm{\sigma}_{kh}||_{\widetilde{0}}+|\textbf{u}-\textbf{u}_{kh}|_{\widetilde{1}}
\lesssim   N^{-s}+ h +
\sum_{n=1}^{N}{\varrho_n}^{-p_n}
\end{equation}
 for any $s>0$ if the  covariance functions of $\widetilde{E}$, $\textbf{f}$ and $\textbf{g}$
are piecewise smooth.
%where $\gamma_n=\max \{
%\sqrt{\widetilde{\lambda}_n}||\widetilde{b}_n||_{L^{\infty}(D)}/{\underline{e}_{N}}, \sqrt{\widehat{\lambda}_n}||\widehat{b}_n||_{0}, \sqrt{\overline{\lambda}_n}||\overline{b}_n||_{0,\Gamma_n} \}$.
\end{thm}
\begin{rem}
This theorem shows the $p\times h$-SHSFEM yields    exponential rates of convergence with
respect to   the degrees $ (p_1,p_2,...,p_N)$  of the polynomials used for approximation.
\end{rem}
%\begin{thm}
%Let $(\bm{\sigma}, \textbf{u})\in (L^2_P(\Omega;~ \Sigma_D), L^2_P(\Omega ;  ~{V_D}))$ ,$(\bm{\sigma}_N, \textbf{u}_N)\in (L^2_{\rho}(\Gamma)\otimes {\Sigma_D}, L^2_{\rho}(\Gamma)\otimes {V_D})$ and $(\bm{\sigma}_{ph}, \textbf{u}_{ph})\in (Z^{\textbf{p}} \otimes {V_D}_h, Z^{\textbf{p}} \otimes {V_D}_h)$ be the solutions of (\ref{continous formulation}),(\ref{truncated continous formulation N}) and (\ref{discrete formulation  ph-version}).  Under the same conditions as in Lemma \ref{lemma assumption}, Lemma \ref{lemma ph 1},  and  if the  covariance functions of $\widetilde{E}$, $\textbf{f}$ and $\textbf{g}$
%are piecewise analytic,  then for any $r>0$
%\begin{equation}
%||\bm{\sigma}-\bm{\sigma}_{kh}||_{\widetilde{0}}+|\textbf{u}-\textbf{u}_{kh}|_{\widetilde{1}}
%\lesssim   e^{-1/2 r N^{1/2}} + h +
%\sum_{n=1}^{N}\varrho_n^{-p_n}.
%\end{equation}
% If the  covariance functions of $\widetilde{E}$, $\textbf{f}$ and $\textbf{g}$
%are piecewise smooth,  then for any $s>1$
%\begin{equation}
%||\bm{\sigma}-\bm{\sigma}_{kh}||_{\widetilde{0}}+|\textbf{u}-\textbf{u}_{kh}|_{\widetilde{1}}
%\lesssim   N^{1-s}+ h +
%\sum_{n=1}^{N}\varrho_n^{-p_n}.
%\end{equation}
%where $\varrho_n=\frac{2 \tau_n}{|\Gamma_n|}+\sqrt{1+\frac{4\tau_n^2}{|\Gamma_n|^2}}$.
%\end{thm}
\begin{figure}[!h]
\begin{center}
\setlength{\unitlength}{0.5cm}
\begin{picture}(22,7)
\put(0,3){\line(1,0){10}} \put(0,3){\line(0,1){2}} \put(0,5){\line(1,0){10}} \put(10,3){\line(0,1){2}} \put(2,3){\line(0,1){2}}
\put(4,3){\line(0,1){2}} \put(6,3){\line(0,1){2}}  \put(8,3){\line(0,1){2}} \put(4,2.4){$5\times 1$}

\put(0,0){\line(1,0){10}} \put(0,0){\line(0,1){2}} \put(0,2){\line(1,0){10}} \put(10,0){\line(0,1){2}} \put(1,0){\line(0,1){2}}
\put(2,0){\line(0,1){2}} \put(3,0){\line(0,1){2}} \put(4,0){\line(0,1){2}} \put(5,0){\line(0,1){2}} \put(6,0){\line(0,1){2}}
\put(7,0){\line(0,1){2}} \put(8,0){\line(0,1){2}} \put(9,0){\line(0,1){2}} \put(0,1){\line(1,0){10}} \put(4,-0.6){$10\times 2$}
\put(2,-2){Rectangular meshes }

\put(15,3){\line(1,0){10}} \put(15,3){\line(0,1){2}} \put(15,5){\line(1,0){10}} \put(25,3){\line(0,1){2}} \put(16,3){\line(1,2){1}}
\put(17,3){\line(1,1){2}} \put(19,3){\line(1,2){1}} \put(22,3){\line(-1,2){1}} \put(19,2.4){$5\times 1$}

\put(15,0){\line(1,0){10}} \put(15,0){\line(0,1){2}} \put(15,2){\line(1,0){10}} \put(25,0){\line(0,1){2}} \put(16,0){\line(1,2){1}}
\put(17,0){\line(1,1){2}} \put(19,0){\line(1,2){1}} \put(22,0){\line(-1,2){1}} \put(15.5,0){\line(1,4){0.5}} \put(16.5,0){\line(3,4){1.5}}
\put(18,0){\line(3,4){1.5}}   \put(20.5,0){\line(0,1){2}} \put(23.5,0){\line(-1,4){0.5}} \put(15,1){\line(1,0){10}} \put(19,-0.6){$10\times 2$}
\put(17,-2){Irregular meshes}
\end{picture}
\end{center}
\vspace{0.5cm} \caption{Finite element meshes }
 \label{fig: FiniteElementMesh}
\end{figure}
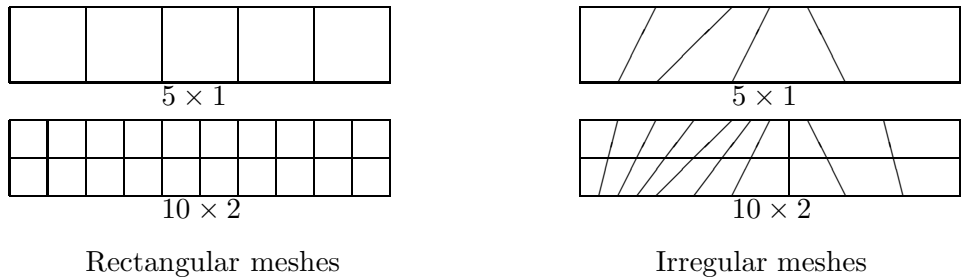

\section{Numerical examples}

In this section we compute two numerical examples  to test the performance of the proposed $p\times h$-version of stochastic hybrid stress finite element method.  We note that the $p\times h$-SHSFEM can be viewed as a particular case of the $k\times h$ version. For convenience we denote 
$$e_u:=\frac{|\textbf{u}-\textbf{u}_{ h}|_{\widetilde{1}}}{|\textbf{u}|_{\widetilde{1}}},\quad e_\sigma:=\frac{||\bm{\sigma}-\bm{\sigma}_{ h}||_{\widetilde{0}}}{||\bm{\sigma}||_{\widetilde{0}}},
$$
where $(\textbf{u}_{ h},\bm{\sigma}_{ h})$ is the corresponding stochastic   finite element approximation to the exact solution $(\textbf{u},\bm{\sigma})$.

\textbf{Example 1} : \textbf{stochastic  plane stress problem}

 Set the spatial domain $D=(0,10)\times (-1,1)$ with meshes as in Figure \ref{fig: FiniteElementMesh}. The body force $\textbf{f}$ and the surface traction $\textbf{g}$ on $\partial D_1=\{(x_1,x_2)\in[0,10]\times[-1,1]:x_1=10 ~\text{or}~ x_2=\pm1\}$ are given by $$\textbf{f}=(0,0)^T, \quad \textbf{g}|_{x_1=10}=(-2\widetilde{E}x_2,0)^T, \quad \textbf{g}|_{x_2=\pm1}=(0,0)^T.$$ The exact solution $(\textbf{u},\bm{\sigma})$ is  of the form
\begin{displaymath}
\textbf{u}=\left( \begin{array}{c} -2x_1x_2 \\ x_1^2+\nu(x_2^2-1)\end{array}\right) ,\quad  \bm{\sigma}=\left(\begin{array}{cc} -2\widetilde{E}x_2 & 0\\0 & 0\end{array} \right),
\end{displaymath}
 where $\widetilde{E}$ is a uniform random variable on $[500,1500]$, and we set $\nu=0.25$.

 In the computation we use the exact form of  the stochastic coefficient  $\widetilde{E}$  and take $N=1$, so there is no   truncation error caused by  the K-L expansion in the approximation.
  Numerical results at different meshes and different values of $p$ are listed in Tables
 \ref{table stress pc  regular}-\ref{table stress pc  irregular}. For comparison we  also list results computed by a stochastic finite element called  $PC\times h$ method, where the polynomial chaos (PC) method \cite{1990 R.G} and the  PS element method are used in the stochastic field $\Gamma$  and  the space domain $D$, respectively.   In the  $PC\times h$ method,  $p$ denotes the degree of polynomial chaos.  We note that  the  computational costs of the $PC\times h$ method and the $p\times h$-SHSFEM are  almost the same with the same $p$.

From  the  numerical results we can see that  the solutions are more accurate with  the increasing   of $p$ and the refinement of meshes. Especially, $p=1$ and $p=2$ for the  $p\times h$-SHSFEM  give almost the
same results, which implies that the solutions are accurate enough with respect to the $p$-version approximation of the stochastic field for given spatial meshes; In these cases,  the  $p\times h$-SHSFEM  is of  first order accuracy  in the mesh size $h$ for  the displacement approximation and yields quite accurate results for the stress approximation.   What's more, we can see that the  $p\times h$-SHSFEM is more accurate than  the $PC\times h$ method at the same $p$.

% the solutions approximated by Legendre-chaos is more efficient for it needs less expansion terms to reach the same
%accuracy  with solutions approximated by  Hermite-chaos.
%This maybe  a natural result for  $\widetilde{E}$ is a uniform variable and Legendre-chaos is its corresponding  generalized polynomial chaos. This  also show  that  generalized polynomial chaos is better than original polynomial chaos for it may  speed  the solving of some problem.

%%%%%%%%%%%%%%%%%%%%%%%%%%%%%%%%%%%%%%%%%%%%%%%%%%%%%%%
%%%%%%%%%%%%%%%%%%%%%%%%%%%%%%%%%%%%%%%%%%%%%%%%%%%%%%%
\begin{table}[!htb]\renewcommand{\baselinestretch}{1.25}\small
 \centering
 %\caption{»ùÓÚPC·½·¨£¬PSÔªÔÚ¹æÔòÍø¸ñϵĽá¹û  }
 \caption{  Results for two methods  under  rectangular meshes: Example 1 }
 \label{table stress pc  regular}
\begin{tabular}{c|cccccccccc}
  \hline
&& &$e_u$&  & & & &$e_\sigma$&  &\\\cline{3-6}\cline{8-11}
 Methods & $p$ & $5\times1$ &$10\times2$ & $20\times4$ & $40\times8$ & &$5\times1$ &$10\times2$ & $20\times4$ & $40\times8$\\
  \hline
   &4&0.0733  &0.0375   &0.0204   & 0.0130 &  &0.0202  &0.0202  &0.0202 &0.0202 \\
    $PC\times h$ &6 &0.0728  &0.0365 &0.0186   &0.0098  &  &0.0079  &0.0079  & 0.0079& 0.0079\\
   &8  &0.0727  & 0.0364  &0.0182   &0.0092  &  &0.0033  &0.0033  &0.0033 &0.0033 \\
   \hline
   & 0& 0.1223  &0.1050   &0.1003  & 0.0990 &  &0.2774  &0.2774  &0.2774 &0.2774 \\
   $p \times h$ &1& 0.0727  &0.0363   & 0.0182  &0.0091   &  &0  &0  &0  &0  \\
    &2&  0.0727  &0.0363   & 0.0182  &0.0091   &  &0  &0  &0  &0  \\
    \hline

\end{tabular}
\end{table}

\begin{table}[!htb]\renewcommand{\baselinestretch}{1.25}\small
 \centering
 %\caption{»ùÓÚPC·½·¨£¬PSÔªÔÚ²»¹æÔòÍø¸ñϵĽá¹û  }
 \caption{Results for two methods under irregular meshes: Example 1}
 \label{table stress pc  irregular}
\begin{tabular}{c|cccccccccc}
  \hline
   %&$\frac{|\textbf{u}-\textbf{u}_h|_{\widetilde{1}}}{|\textbf{u}|_{\widetilde{1}}}$ &  &  &  &  %&$\frac{||\bm{\sigma}-\bm{\sigma}_h||_{\widetilde{0}}}{||\bm{\sigma}||_{\widetilde{0}}}$  &  & & \\
&& &$e_u$&  & & & &$e_\sigma$&  &\\\cline{3-6}\cline{8-11}
 Methods & $p$ & $5\times1$ &$10\times2$ & $20\times4$ & $40\times8$ & &$5\times1$ &$10\times2$ & $20\times4$ & $40\times8$\\
  \hline
  &4&0.1431 & 0.0637  & 0.0325  &0.0181  &  &0.2632  &0.0579  &0.0231 &0.0203 \\
   $PC \times h$ &6 &0.1429  &0.0631  &0.0314   &0.0160  &  &0.2626  &0.0549  &0.0137 &0.0083 \\
   &8  &0.1429  &0.0630   &0.0312   &0.0156  &  &0.2625  & 0.0544  &0.0117&0.0041 \\
   \hline
   & 0& 0.1435 &0.1160  &0.1037  &0.0999  &  &0.3684  &0.2816  &0.2775 &0.2774 \\
   $p\times h$ & 1& 0.1429 &0.0630   & 0.0311  &0.0155   &  & 0.2524 &0.0509   &0.0104  &0.0023  \\
   & 2&  0.1429 &0.0630   & 0.0311  &0.0155   &  & 0.2524 &0.0509   &0.0104  &0.0023    \\
   \hline
\end{tabular}
\end{table}

\textbf{Example 2}  :    \textbf{stochastic plane strain problem}

The  domain $\Omega$ and  meshes are the same as in Figure \ref{fig: FiniteElementMesh}.
The body force $\textbf{f}=(0,0)^T$. The surface traction $\textbf{g}$ on $\partial D_1=\{(x_1,x_2)\in[0,10]\times[-1,1]:x_1=10 ~or~ x_2=\pm1\}$ is given by $\textbf{g}|_{x_1=10}=(-2\widetilde{E}x_2,0)^T$, $\textbf{g}|_{x_2=\pm1}=(0,0)^T$, and the exact solution $(\textbf{u},\bm{\sigma})$ is  of the form
\begin{displaymath}
\textbf{u}=\left( \begin{array}{c} -2(1-\nu^2)x_1x_2 \\ (1-\nu^2)x_1^2+\nu(1+\nu)(x_2^2-1)\end{array}\right),~~~~~~~~~~ ~~~ \bm{\sigma}=\left(\begin{array}{cc} -2\widetilde{E}x_2 & 0\\0 & 0\end{array} \right),
\end{displaymath}
 where $\widetilde{E}=1+\xi^2$, $\xi$ is a  standard  normal Gaussian random variable.

 Similar to Example 1, in the computation we use the exact form of  the stochastic coefficient  $\widetilde{E}$ and take $N=1$.
  Numerical results at different meshes, different values of $p$ and different values of Poisson ratio $\nu$ are listed in Tables \ref{tab3}-\ref{tab8}.    For comparison we  also list results computed by a stochastic finite element called  $p \times $bilinear method, where the $p$-version method  and the  bilinear element  are used in the stochastic field $\Gamma$  and  the space domain $D$, respectively.     We note that  the  computational costs of the $p \times $bilinear method and the $p\times h$-SHSFEM are  almost the same.

Tables \ref{tab3}-\ref{tab4} show that the $p \times $bilinear method deteriorates as $\nu\rightarrow0.5$ or $\lambda\rightarrow +\infty$, while Tables \ref{tab5}-\ref{tab8} show that the $p\times h$-SHSFEM yields uniformly accurate results for  the displacement and stress approximations. Moreover,  $p=0$ and $p=2$ give almost the same results, which implies that the solutions are  accurate enough with respect to the $p$-version approximation of the stochastic field for given spatial meshes.

\begin{table}[!htb]\renewcommand{\baselinestretch}{1.25}\small
 \centering
%\caption{M=1ʱ˫ÏßÐÔÔªµÄÎó²î½á¹û£¨$\frac{|\textbf{u}-\textbf{u}_h|_{\widetilde{1}}}{|\textbf{u}|_{\widetilde{1}}}$£©  }
 \caption{Results of $e_u$ for Example 2: $p\times$bilinear  method, $p=0$ }
 \label{tab3}
 \begin{tabular}{ccccccccccc}
   \hline
    && Rectangular  & meshes &  & && Irregular  & meshes &\\\cline{2-5}\cline{7-10}
  $\nu$&10$\times$2& 20$\times$4 & 40$\times$8 &80$\times$16  &~ & 10$\times$2 & 20$\times$4 & 40$\times$8 &80$\times$16 \\
   \hline
    0.25& 0.5384  &  0.3061 &    0.1625  &0.0883  &  &  0.6854  &  0.4501    & 0.2532 &0.1356\\
    0.49 &0.8516  &  0.6523  &  0.4034 &    0.2175 &¡¡&   0.8782  &  0.7424 &   0.5218 &0.3038\\
   0.499& 0.9533  &  0.9070   & 0.7856  &  0.5579 & &   0.9511 &   0.9145   & 0.8322  &0.6617\\
  0.4999&  0.9661 &   0.9556 &   0.9365  &  0.8760  & &  0.9641  &  0.9550  &  0.9378 &0.8925  \\
   \hline
 \end{tabular}
\end{table}

\begin{table}[!htb]\renewcommand{\baselinestretch}{1.25}\small
 \centering
% \caption{M=3ʱ˫ÏßÐÔÔªµÄÎó²î½á¹û£¨$\frac{|\textbf{u}-\textbf{u}_h|_{\widetilde{1}}}{|\textbf{u}|_{\widetilde{1}}}$£©  }
 \caption{Results of $e_u$ for Example 2: $p\times$bilinear method, $p=2$}
 \label{tab4}
\begin{tabular}{ccccccccccc}
  \hline
   %$\frac{|\textbf{u}-\textbf{u}_h|_{\widetilde{1}}}{|\textbf{u}|_{\widetilde{1}}}$  & $Regular$ &$mesh$& ~ &~ &~&$Irregular$&$mesh$&~&~ &~\\
   && Rectangular & meshes&  & && Irregular & meshes &\\\cline{2-5}\cline{7-10}
  $\nu$&10$\times$2& 20$\times$4 &40$\times$8 &80$\times$16  &~ & 10$\times$2 & 20$\times$4 & 40$\times$8 &80$\times$16 \\
  \hline
    0.25&    0.5384&    0.3061  &  0.1625 &  0.0883  &&  0.6854  &  0.4501  &  0.2532 & 0.1356\\
    0.49&    0.8516  &  0.6523 &   0.4034  &  0.2175  &&  0.9511 &   0.9145 &   0.8322  &0.6617\\
    0.499&   0.9533 &   0.9070 &   0.7856&    0.5579 &&   0.9511  &  0.9145&    0.8322 & 0.6617\\
    0.4999&    0.9661 &   0.9556  &  0.9365 &   0.8760 &&   0.9641  &  0.9550  &  0.9378 & 0.8925\\

  \hline
\end{tabular}
\end{table}

\begin{table}[!htb]\renewcommand{\baselinestretch}{1.25}\small
 \centering
 %\caption{M=1ʱPSÔªµÄÎó²î½á¹û£¨$\frac{|\textbf{u}-\textbf{u}_h|_{\widetilde{1}}}{|\textbf{u}|_{\widetilde{1}}}$£©  }
 \caption{Results of $e_u$  for Example 2:   $p\times h$ SHSFEM,  $p=0$}
 \label{tab5}
\begin{tabular}{ccccccccccc}
  \hline
   %$\frac{|\textbf{u}-\textbf{u}_h|_{\widetilde{1}}}{|\textbf{u}|_{\widetilde{1}}}$  & $Regular$ &$mesh$& ~ &~ &~&$Irregular$&$mesh$&~&~ &~\\
    && Rectangular & meshes &  & &&Irregular & meshes &\\\cline{2-5}\cline{7-10}
  $\nu$&10$\times$2& 20$\times$4 & 40$\times$8 &80$\times$16  &~ & 10$\times$2 & 20$\times$4 & 40$\times$8 &80$\times$16 \\
  \hline
  0.25&  0.0372 &0.0186  &0.0093  &0.0046&~&0.0676&0.0323&0.0158&0.0079\\
  0.49& 0.0488  &0.0244  & 0.0122 &0.0061&~&0.0763&0.0371&0.0183&0.0091\\
  0.499& 0.0497  &0.0248  &0.0124  &0.0062&~&0.0770&0.0375&0.0185&0.0092\\
  0.4999&  0.0497 &0.0249  &0.0124  &0.0062&~&0.0770&0.0375&0.0185&0.0092\\
  \hline
\end{tabular}
\end{table}

\begin{table}[!htb]\renewcommand{\baselinestretch}{1.25}\small
 \centering
% \caption{M=1ʱPSÔªµÄÎó²î½á¹û£¨$\frac{||\bm{\sigma}-\bm{\sigma}_h||_{\widetilde{0}}}{||\bm{\sigma}||_{\widetilde{0}}}$£©  }
\caption{ Results of $e_\sigma$  for Example 2:  $p\times h$ SHSFEM,  $p=0$}
 \label{tab6}
\begin{tabular}{ccccccccccc}
  \hline
 % $\frac{||\bm{\sigma}-\bm{\sigma}_h||_{\widetilde{0}}}{||\bm{\sigma}||_{\widetilde{0}}}$ & $Regular$ &$mesh$& ~ &~ &~&$Irregular$&$mesh$&~&~ &~\\
 && Rectangular  & meshes&  & && Irregular & meshes &\\\cline{2-5}\cline{7-10}
   $\nu$&10$\times$2& 20$\times$4 & 40$\times$8 &80$\times$16  &~ & 10$\times$2 & 20$\times$4 & 40$\times$8 &80$\times$16 \\
  \hline
  0.25 & 0  &0&0   &0  &  &  0.1513  &0.0866  &0.0450  &0.0227   \\
   0.49 & 0&   0&  0&0   &  & 0.1559&0.0877 &0.0451 &0.0227 \\
   0.499&  0&0  &0  &0  &   &0.1563 & 0.0878 &0.0452  &0.0227  \\
   0.4999& 0 &  0& 0 & 0 &   &0.1564  &0.0878  &0.0452  &0.0227  \\
  \hline
\end{tabular}
\end{table}

\begin{table}[!htb]\renewcommand{\baselinestretch}{1.25}\small
 \centering
% \caption{M=3ʱPSÔªµÄÎó²î½á¹û£¨$\frac{|\textbf{u}-\textbf{u}_h|_{\widetilde{1}}}{|\textbf{u}|_{\widetilde{1}}}$ £© }
  \caption{Results of $e_u$  for Example 2:   $p\times h$ SHSFEM,  $p=2$}
 \label{tab7}
\begin{tabular}{ccccccccccc}
  \hline
   %$\frac{|\textbf{u}-\textbf{u}_h|_{\widetilde{1}}}{|\textbf{u}|_{\widetilde{1}}}$  & $Regular$ &$mesh$& ~ &~ &~&$Irregular$&$mesh$&~&~ &~\\
 && Rectangular & meshes&  & && Irregular & meshes &\\\cline{2-5}\cline{7-10}
  $\nu$&10$\times$2& 20$\times$4 & 40$\times$8 &80$\times$16  &~ & 10$\times$2 & 20$\times$4 & 40$\times$8 &80$\times$16 \\
  \hline
  0.25& 0.0372 &0.0186  &0.0093  &0.0046  &  &0.0676  &0.0323  &0.0158  &0.0079  \\
   0.49& 0.0488 &0.0244  &0.0122  &0.0061 &  &0.0763  &0.0371  &0.0183  &0.0091 \\
   0.49& 0.0497 &0.0248  &0.0124  &0.0062  &  &0.0770  &0.0375  &0.0185  &0.0092  \\
   0.4999&0.0497  &0.0249&0.0124 &0.0062  &  &0.0770  &0.0375  &0.0185  &0.0092  \\
  \hline
\end{tabular}
\end{table}

\begin{table}[!htb]\renewcommand{\baselinestretch}{1.25}\small
 \centering
 %\caption{M=3ʱPSÔªµÄÎó²î½á¹û£¨$\frac{||\bm{\sigma}-\bm{\sigma}_h||_{\widetilde{0}}}{||\bm{\sigma}||_{\widetilde{0}}}$ £© }
 \caption{Results of $e_\sigma$  for Example 2:   $p\times h$ SHSFEM,  $p=2$}
 \label{tab8}
\begin{tabular}{ccccccccccc }
  \hline
  % $\frac{||\bm{\sigma}-\bm{\sigma}_h||_{\widetilde{0}}}{||\bm{\sigma}||_{\widetilde{0}}}$  & $Regular$ &$mesh$& ~ &~ &~&$Irregular$&$mesh$&~&~ &~\\
  && Rectangular & meshes&  & && Irregular & meshes&\\\cline{2-5}\cline{7-10}
   $\nu$&10$\times$2& 20$\times$4 & 40$\times$8 &80$\times$16  &~ & 10$\times$2 & 20$\times$4 & 40$\times$8 &80$\times$16 \\
   \hline
 0.25& 0 &0 &0  &0  &  & 0.1513 &0.0866  &0.0450  &0.0227  \\
  0.49& 0 &0  &0  &0  &  & 0.1559& 0.0877 & 0.0451 & 0.0227 \\
   0.499&  0&0  &0  &0 &  & 0.0156 &0.0878  &0.0452  &0.0227  \\
   0.4999& 0 &  0&  0& 0 & & 0.1564 & 0.0878 &0.0452  &0.0277  \\
  \hline
\end{tabular}
\end{table}

\clearpage
%%%%%%%%%%%%%%%%%%%%%%%%%%%%%%%%%%%%%%%%%%%%%%%%%%%%%%%
%%%%%%%%%%%%%%%%%%%%%%%%%%%%%%%%%%%%%%%%%%%%%%%%%%%%%%%

%%%%%%%%%%%%%%%%%%%%%%%%%%%%%%%%%%%%%%%%%%%%%%%%%%%%%
\end{document}